\documentclass[12pt]{amsart}
\selectfont

\usepackage{amssymb,color, euscript, enumerate}
\usepackage{amsthm}
\usepackage{amsmath}
\usepackage{amscd}
\usepackage{txfonts}
\usepackage{comment}
\usepackage{bm}
\usepackage{amscd}
\usepackage{empheq}
\usepackage{cases}

\newtheorem{thm}{Theorem}[section]
\newtheorem{dfn}[thm]{Definition}
\newtheorem{lem}[thm]{Lemma}
\newtheorem{prop}[thm]{Proposition}
\newtheorem{cor}[thm]{Corollary}
\newtheorem{rem}[thm]{Remark}

\begin{document}
\title[Genus of Vertex Algebras and Mass Formula]{Genus of Vertex Algebras and Mass Formula}
\author{YUTO MORIWAKI}

\begin{center}
{\LARGE \bf Genus of vertex algebras and mass formula
}
 \par \bigskip

\renewcommand*{\thefootnote}{\fnsymbol{footnote}}
{\normalsize
Yuto Moriwaki \footnote{email: \texttt{yuto.moriwaki@ipmu.jp}}
}
\par \bigskip
{\footnotesize Graduate School of Mathematical Science, The University of Tokyo, 3-8-1
Komaba, Meguro-ku, Tokyo 153-8914, Japan}

\par \bigskip
\end{center}

\newcommand{\Z}{\mathbb{Z}}
\newcommand{\1}{\bm{1}}
\newcommand{\C}{\mathbb{C}}
\newcommand{\al}{\alpha}
\newcommand{\be}{\beta}
\newcommand{\ga}{\gamma}
\newcommand{\Hv}{M_H (1,0)}
\newcommand{\Ht}{\tilde{H}}
\newcommand{\Lt}{\tilde{L}}
\newcommand{\Aut}{\mathrm{Aut}\,}

\newcommand{\Hom}{\mathrm{Hom}\,}
\newcommand{\VH}{\underline{\text{VH pair}}}
\newcommand{\AH}{\underline{\text{AH pair}}}
\newcommand{\lat}{\mathrm{lat}}
\newcommand{\mass}{\mathrm{mass}}

\newcommand{\gen}{\mathrm{genus}}

\newcommand{\Wt}{\tilde{W}}
\newcommand{\Vt}{{\tilde{V}}}
\newcommand{\Mt}{{\tilde{M}}}
\newcommand{\genus}{{{I\hspace{-.1em}I}}}
\newcommand{\tw}{{{I\hspace{-.1em}I}_{1,1}}}
\newcommand{\six}{{{I\hspace{-.1em}I}_{25,1}}}
\newcommand{\End}{\mathrm{End}}
\newcommand{\Sumz}{\bigoplus_{n \in \mathbb{Z}}}
\newcommand{\Sumn}{\bigoplus_{n \geq 0}}
\newcommand{\fora}{\text{ for any }}

\vspace*{8mm}

\noindent
\textbf{Abstract.}
We introduce the notion of a genus and its mass for vertex algebras. For lattice vertex algebras,
their genera are the same as those of lattices, which play an important role in the classification of lattices.
We derive a formula relating the mass for vertex algebras to that for lattices,
and then give a new characterization of some holomorphic vertex operator algebras.
\vspace*{8mm}

\section*{Introduction}
%
%
Vertex operator algebras (VOAs) are algebraic structures that play an important role in two-dimensional conformal field theory.
The classification of conformal field theories is an extremely interesting
problem, of importance in mathematics, statistical mechanics, and string theory;
Mathematically, it is related to the classification of vertex operator algebras,
which is a main theme of this work. An important class of vertex operator algebras can be constructed from positive-definite even integral lattices, called {\it lattice vertex algebras} 
\cite{Bo1,LL,FLM}.
In this paper, we propose a new method to construct and classify vertex algebras by
using a method developed in the study of lattices.

A lattice is a finite rank free abelian group  equipped with a $\Z$-valued symmetric bilinear form. Two such lattices are equivalent (or {\it in the same genus}) if they are isomorphic over the ring of $p$-adic integers $\Z_p$ for each prime $p$ and also equivalent over the real number field $\mathbb{R}$.
Lattices in the same genus are not always isomorphic over $\Z$.
For example, there are two non-isomorphic rank 16 positive-definite even unimodular lattices $D_{16}^+$ and $E_8\oplus E_8$ which are in the same genus. {\it The mass of the genus of a lattice} $L$ is defined by
$\mass(L)=\sum_{M}1/{|\Aut M|},$ where the sum is over all isomorphism classes of lattices in the genus.
The mass can be calculated by the Smith-Minkowski-Siegel mass formula (see for example \cite{CS,Ki}). The concept of genera of lattices and the mass formula are important for the classification of lattices. In fact, rank $24$ positive-definite even unimodular lattices, called Niemeier lattices, can be classified by using the mass formula.

The following result is important to define the notion of genus of vertex algebras; Two integral lattices $L_1$ and $L_2$ are in the same genus if and only if $L_1 \oplus \tw$ and $L_2 \oplus \tw$ are isomorphic,
where $\tw$ is the unique even unimodular lattice of signature $(1,1)$ (\cite{CS}).
Motivated by this fact, we would like to define that two vertex algebras $V_1$ and $V_2$ are in the same genus if $V_1 \otimes V_\tw$ and 
$V_2 \otimes V_\tw$ are isomorphic.
Little is known about the structure of $V\otimes V_\tw$, 
because we do not have many tools to analyze $\Z$-graded vertex algebras whose homogeneous subspaces are infinite dimensional.
%
We overcome this difficulty by using a Heisenberg vertex subalgebra as an
additional structure.
To be more precise, we consider a pair of a vertex algebra $V$ and its Heisenberg vertex subalgebra $H$, called a VH pair $(V,H)$, and define the notions of genus and mass of VH pairs, and prove
a formula relating the mass of VH pairs and that of lattices, which we call here a mass formula for vertex algebras (Theorem \ref{mass_vertex}).
The important point to note here is
that if VH pairs $(V,H_V)$ and $(W,H_W)$ are in the same genus, then 
the module categories $\underline{V\text{-mod}}$ and $\underline{W \text{-mod}}$
are equivalent (Theorem \ref{coset_equivalence}). Hence, if $V$ has a good representation theory (e.g., completely reducible), then $W$ also has a good representation theory.
This result suggests that the genus of VH pairs can be used to construct
good vertex algebras from a good vertex algebra.

The tensor product functor $- \otimes V_\tw$ is first considered by R. Borcherds and is
an important step in his proof of the monstrous moonshine conjecture (see \cite{Bo2}).
It is also used by G. H{\"o}hn and N. Scheithauer in their constructions of holomorphic VOAs \cite{HS1}.
In the process of the constructions, they show that some non-isomorphic $17$ holomorphic VOAs of central charge $24$ are isomorphic to each other after taking the tensor product $-\otimes V_\tw$, that is, in the same genus in our sense \cite{HS1}.
They also study other examples of holomorphic VOAs with non-trivial genus (see \cite{H2}, \cite{HS2}),
which motivates us to define the genus.

We remark that G. H{\"o}hn gives another definition of a genus of vertex operator algebras based on the representation theory of VOAs (more precisely, modular tensor category) \cite{H1}.
All holomorphic VOAs of central charge $24$ are in the same genus in their definition,
whereas they are not so in our sense.
Following the pattern suggested by the equivalence of categories in Theorem \ref{coset_equivalence}, we believe
that if VOAs are in the same genus in our sense, then
they are in the same genus in the sense of \cite{H1}.

\vspace{3mm}

{\it Organization of the paper}\quad
Let us explain our basic idea and the contents of this paper.
It is worth pointing out that the construction of lattice vertex algebras from
even lattices, $L \mapsto V_L$, is not a functor from the category of lattices to
the category of vertex algebras (see Remark \ref{lattice_functor}).
We focus on a twisted group algebra $\C\{L\}$ considered in \cite{FLM} which
plays an important role in the construction of lattice vertex algebras.
In Section \ref{sec_AH}, we generalize the twisted group algebra
and introduce the notion of an AH pair $(A,H)$,
which is a pair of an associative algebra $A$ satisfying some properties and a finite dimensional vector space $H$ equipped with a bilinear form (see Section \ref{AH}).
Then, we construct a functor from the category of AH pairs
to the category of VH pairs,
$\mathcal{V}: \AH \rightarrow \VH,\;(A,H) \mapsto (V_{A,H},H)$.
In the special case where an AH pair is the twisted group algebra 
$(\C\{L\},\C\otimes_\Z L)$ constructed from an even lattice $L$,
its image by this functor coincides with the lattice vertex algebra $V_L$.
We call a class of AH pairs which axiomatizes twisted group algebras a {\it lattice pair},
and 
the vertex algebra associated with a lattice pair a {\it $H$-lattice vertex algebra.} Lattice pairs are classified by using the cohomology of an abelian group.

In Section \ref{sec_VH}, we construct a right adjoint functor of $\mathcal{V}$, denoted by $\Omega:\VH \rightarrow \AH,\; (V,H)\mapsto (\Omega_{V,H},H)$ and, in particular, construct an associative algebra from
a vertex algebra (Theorem \ref{adjoint}).
In the case that a VH pair $(V,H)$ is a simple as a vertex algebra,
the structure of the AH pair $(\Omega_{V,H},H)$ is studied by using the results in the previous section. We prove that there exists a
maximal $H$-lattice vertex subalgebra (Theorem \ref{maximal_lattice}),
which implies we can classify $H$-lattice vertex subalgebras in a VH pair $(V,H)$.
This maximal algebra is described by a lattice $L_{V,H} \subset H$, called a maximal lattice of the good VH pair $(V,H)$.


In Section \ref{sec_genus}, we define the notion of a genus of VH pairs.
We remark that the category of VH pairs (resp. AH pairs) naturally admits
a monoidal category structure
and any lattice vertex algebra is naturally a VH pair.
Two VH pairs $(V,H_V)$ and $(W,H_W)$ are said to be in the same genus if the
VH pairs $(V\otimes V_\tw,H_V\oplus H_\tw)$ and $(W \otimes V_\tw, H_W\oplus H_\tw)$ are isomorphic.
Roughly speaking, classifying VH pairs in the genus of $(V,H_V)$ is equivalent 
to classifying subalgebras in $V\otimes V_\tw$ which are isomorphic to $V_\tw$
and compatible with the VH pair structure;
We prove that such subalgebras are completely described by
the maximal lattice $L_{V\otimes V_\tw} \subset H_V \oplus H_\tw$.
In this way the classification of genera of VH pairs is reduced to the study of maximal lattices.
For a VH pair $(V,H_V)$,
a subgroup of the automorphism group of the lattice, $G_{V,H_V} \subset \Aut L_{V,H_V}$,
is defined from the automorphism group of a vertex algebra.
The difference between a genus of lattices and a genus of VH pairs is measured
by this subgroup $G_{V,H_V}$.
The mass of the VH pairs in the genus which contains $(V,H_V)$ is defined by $\mass(V,H_V)=\sum_{(W,H_W)} \frac{1}{|G_{W,H_W}|}$,
where the sum is over all the VH pairs $(W,H_W)$ in the genus,
which is finite if $L_{V,H_V}$ is positive-definite and the index of groups
$[\Aut L_{V,H_V}\oplus \tw: G_{V\otimes V_\tw,H_v\oplus H_\tw}]$ is finite.
In this case, we prove that the ratio of the masses $\mass(V,H_V)/ \mass(L_{V,H_V})$ is equal to
the index of groups $[\Aut L_{V,H_V}\oplus \tw: G_{V\otimes V_\tw,H_v\oplus H_\tw}]$ (Theorem \ref{mass_vertex}).

In the final section, a conformal vertex algebra $(V,\omega)$ with a Heisenberg vertex subalgebra $H$ is studied.
This algebra is graded by the action of the Virasoro algebra and the Heisenberg Lie algebra,
$V=\bigoplus_{\al \in H, n \in \Z} V_n^\al$.
We assume here technical conditions, e.g., $V_n^\al=0$ if $n<\frac{(\al,\al)}{2}$ (see Definition \ref{Cartan}),
which are satisfied in interesting cases, e.g., extensions of 
rational affine vertex algebras.
Those conformal vertex algebras are closed under the tensor product, which means that it is possible to study their genus.
The main results of this section are (Theorem \ref{Cartan_max}, Theorem \ref{Cartan_Lie})
\begin{enumerate}
\item
$\dim V_n^\al =0, 1$ if $n \geq \frac{(\al,\al)}{2} > n-1$;
\item
The maximal lattice $L_{V,H}$ is equal to $\{\al \in H\;|\; V_{\frac{(\al,\al)}{2}}^\al \neq 0 \}$;
\item
If $V_1^\al \neq 0$ for $(\al,\al)>0$,
then $V_1^\al$ and $V_1^{-\al}$ generate the simple affine vertex algebra $L_{\mathrm{sl}_2}(k,0)$
with $k=\frac{2}{(\al,\al)} \in \Z_{>0}$.
\end{enumerate}
The results suggest that a vector in $V_n^\al$ with $n \geq \frac{(\al,\al)}{2} > n-1$ has interesting structure,
which is beyond the scope of this paper.
As an application, we study the genus of some holomorphic VOAs.
A {\it holomorphic VOA} (or more generally a {\it holomorphic conformal vertex algebra}) is a VOA such that any module is completely reducible and it has a unique irreducible module.
The holomorphic VOA $V_{E_{8,2}B_{8,1}}^{hol}$ constructed in \cite{LS2} is uniquely characterized
as a holomorphic VOA of central charge $24$ whose Lie algebra, the degree one subspace of the VOA, is
the semisimple Lie algebra $\mathfrak{g}_{E_8}\oplus \mathfrak{g}_{B_8}$ \cite{LS}.
Let $H$ be a Cartan subalgebra of the degree one subspace. Then, the maximal lattice $L_{V_{E_{8,2}B_{8,1}}^{hol},H}$
is $\sqrt{2}E_8\oplus D_8$,  and we determine the genus of this VOA, which gives another proof of the above mentioned result of \cite{HS1}.
Furthermore, by using the characterization,
we prove that if a holomorphic conformal vertex algebra of central charge $26$ satisfies
our condition and its maximal lattice is $\sqrt{2}E_8\oplus D_8\oplus \tw$,
then it is isomorphic to $V_{E_{8,2}B_{8,1}}^{hol}\otimes V_\tw$ (Theorem \ref{char_18}).

%



%

\tableofcontents

\section{Preliminaries}
We denote the sets of all integers, positive integers,
real numbers and complex numbers by $\Z$, $\Z_{>0}$, $\mathbb{R}$ and $\C$ respectively.
This section provides definitions and notations we need for what follows.
Most of the contents in subsection \ref{subsec_vertex} and \ref{subsec_lattice} are taken from the literature \cite{Bo2,LL,Li,FHL,FLM,CS}.
Throughout this paper, we will work over the field $\C$ of complex numbers.
\subsection{Vertex algebras and their modules} \label{subsec_vertex}
%
Following \cite{Bo2}, a vertex algebra is a $\C$-vector space
$V$ equipped with a linear map 
\begin{align}
Y(-,z):V \rightarrow \End(V)[[z^{\pm}]],\;
a\mapsto Y(a,z)=\sum_{n\in\Z}a(n)z^{-n-1} \nonumber
\end{align}
and a non-zero vector $\1 \in V$ satisfying the following axioms:
\begin{enumerate}
\item[V1)]
For any $a,b \in V$, there exists $N \in \Z$ such that $a(n)b=0$ for any $n \geq N$;
\item[V2)]
For any $a \in V$,
\[
  a(n)\1 = \begin{cases}
    0, & (n \geq 0), \\
    a, & (n=-1),
  \end{cases}
\] holds;

\item[V3)]
Borcherds identity:
For any $a,b \in V$ and $p,q,r \in \Z$,
\begin{equation}
\sum_{i=0}^\infty{p\choose i}(a(r+i)b)(p+q-i)
=\sum_{i=0}^\infty (-1)^i \binom{r}{i}\Bigl(a(p+r-i)b(q+i)-(-1)^r b(q+r-i)a(p+i)\Bigr) \nonumber
\label{borcherds}
\end{equation}
holds.
\end{enumerate}

For $a,b \in V$, $a(n)b$ is called the $n$-th product of $a$ and $b$.
Let $T_V$ (or simply $T$) denote the endomorphism of $V$ defined
by $T_V a = a(-2)\1$ for $a \in V$.

The following properties follow directly from the axioms of a vertex algebra (see for example \cite{LL}):
\begin{enumerate}

\item
$Y(\1,z)=\mathrm{id}_V \label{vac_id};$
\item
For any $a,b \in V$ and $n \in \Z$, $T\1=0$ and
\begin{align}
(Ta)(n)&=-na(n-1), \nonumber \\
T(a(n)b)&=(Ta)(n)b+a(n)Tb \nonumber ;
\end{align}

\item
Skew-symmetry:
For any $a,b \in V$ and $n \in \Z$,

$$a(n)b=\sum_{i \geq 0}(-1)^{n+1+i}\frac{T^i}{i!}(b(n+i)a); \label{skew-symmetry}$$
\item Associativity:
For any $a,b \in V$ and $n,m \in \Z$,
$$(a(n)b)(m)=\sum_{i \geq 0}\binom{n}{i}(-1)^i(a(n-i)b(m+i)-(-1)^nb(n+m-i)a(i)); \label{eq_associative}$$

\item Commutativity:
For any $a,b \in V$ and $n,m \in \Z$,
$$[a(n),b(m)]=\sum_{i \geq 0}\binom{n}{i} (a(i)b)(n+m-i). \label{commutative}$$
\end{enumerate}
We denote by \underline{$V$-mod} the category of $V$-modules.

\subsection{Lattice} \label{subsec_lattice}
A {\it lattice} of rank $n \in \mathbb{N}$ is a rank $n$ free abelian group $L$
equipped with a $\mathbb{Z}$-valued symmetric bilinear form
$$( \,\;,\; ):L \times L \rightarrow \mathbb{Q}.$$
A lattice $L$ is said to be {\it even} if
$$(\al,\al) \in 2\Z \fora \al \in L,$$
{\it integral} if
$$(\al,\be) \in \Z \fora \al ,\be \in L,$$
and {\it positive-definite} if
$$(\al,\al)>0 \fora \al \in L\setminus \{0\}.$$
For an integral lattice $L$ and a unital commutative ring $R$,
we extend the bilinear form $(\;\,,\;)$ bilinearly to $L\otimes_{\Z} R$
and $L$ is said to be {\it non-degenerate} if the bilinear form on $L\otimes_\Z \mathbb{R}$
is non-degenerate.
The dual of $L$ is the set 
$$L^\vee=\{\al \in L\otimes_{\Z}\mathbb{R} \;|\;(\al,L)\subset \Z \}.$$
The integral lattice $L$ is said to be unimodular if
$L= L^\vee$.

Two integral lattices $L$ and $M$ are said to be equivalent or {\it in the same genus} if their base changes are isomorphic as lattices:
$$L\otimes_{\Z}\mathbb{R}\simeq M\otimes_{\Z}\mathbb{R},\quad L\otimes_{\Z}\Z_p\simeq M\otimes_{\Z}\Z_p,$$
for all the prime integers $p$,
where $\Z_p$ is the ring of $p$-adic integers.
Denote by $\gen(L)$ the genus of lattices which contains $L$.
If $L$ is positive-definite, then a mass of its genus $\mass(L)\in \mathbb{Q} $ is defined by
\begin{align}\label{eq: mass formula for lattices}
\mass(L) = \sum_{L' \in \gen(L)} \frac{1}{|\Aut L'|},
\end{align}
where $\Aut L$ is the automorphism group of $L$.
The Smith-Minkowski-Siegel mass formula is a formula which computes $\mass(L)$ (see \cite{CS,Ki}).
Lattices over $\mathbb{R}$ are completely determined by the signature.
Similarly, lattices over $\Z_p$ are determined by some invariant, called $p$-adic signatures
(If $p=2$, we have to consider another invariant, called oddity).
Conventionally, a genus of lattices is denoted by using those invariants,
e.g., the genus of $\sqrt{2}E_8 D_8$ is denoted by $\genus_{16,0}(2_\genus^{+10})$ (see \cite{CS} for the precise definition).

\section{AH pairs and vertex algebras} \label{sec_AH}
%
%
In this section, we generalize the famous construction of vertex algebras from even lattices. 

In Section \ref{cohomology}, we recall some fundamental results 
from the cohomology of abelian groups.
Much of the material in this subsection is based on \cite[Chapter 5]{FLM}. We focus on a twisted group algebras considered in \cite{FLM} and generalize their
result.
In Section \ref{AH}, we introduce the concept of an AH pair which is a generalization of a twisted group algebra, while in Section \ref{AH_vertex}, we construct a vertex algebra associated with an AH pair generalizing the construction of lattice vertex algebras.

\subsection{Twisted group algebras}\label{cohomology}


Let $A$ be an abelian group.
A (normalized) 2-cocycle of $A$ with coefficients in 
$\C^\times$ is a map $f: A\times A \rightarrow \C^\times$ such that:
\begin{enumerate}
\item
$f(0,a)=f(a,0)=1$
\item
$f(a,b)f(a+b,c)=f(a,b+c)f(b,c).$
\end{enumerate}
We denote by $Z^2(A,\C^\times)$ the set of the 2-cocyles of $A$ with coefficients in 
$\C^\times$.
For a map $h:A \rightarrow \C^\times$, the map
$$dh: A \times A \rightarrow \C^\times,\; (a,b)\mapsto h(a+b)h(a)^{-1}h(b)^{-1}$$ is a 2-cocycle, called a coboundary.
The subset of  $Z^2(A,\C^\times)$ consisted of coboundaries is denoted by 
$B^2(A,\C^\times)$, which is a subgroup of $Z^2(A,\C^\times)$. Set $H^2(A,\C^\times)=Z^2(A,\C^\times)/B^2(A,\C^\times)$.
An alternative bilinear form on an abelian group $A$ is
a map $c:A \times A \rightarrow \C^\times$  such that:
\begin{enumerate}
\item
$c(0,a)=c(a,0)=1$ for any $a \in A$;
\item
$c(a+b,c)=c(a,c)c(b,c)$
and $c(a,b+c)=c(a,b)c(a,c)$
 for any $a,b,c \in A$;
\item
$c(a,a)=1$ for any $a \in A$.
\end{enumerate}
The set of alternative bilinear forms on $A$ is denoted by 
$Alt^2(A)$.
For a 2-cocycle $f \in Z^2(A,\C^\times)$,
define the map $c_f:A\times A \rightarrow \C^\times$
 by setting $c_f(a,b)=f(a,b)f(b,a)^{-1}$.
It is easy to prove that the map $c_f$ is an alternative bilinear form on $A$,
which is called a {\it commutator map} \cite{FLM}.
Hence, we have a map $Z^2(A,\C^\times) \rightarrow Alt^2(A)$,
which induces
$c: H^2(A,\C^\times) \rightarrow Alt^2(A)$.
\begin{lem}
\label{vanish}
The map $c: H^2(A,\C^\times) \rightarrow Alt^2(A)$ is injective.
\end{lem}
\begin{proof}
For any 2-cocycle $f \in Z^2(A,\C^\times)$, there exists a central extension
of $A$ by $\C^\times$,
$$1\rightarrow \C^\times \rightarrow \tilde{A} \rightarrow A \rightarrow 1$$
and $f$ is a coboundary if and only if the extension splits (see \cite{FLM}).
If $c_f=1$, that is, $f(a,b)=f(b,a)$, then $\tilde{A}$ is an abelian group and the above sequence is an exact sequence in the category of abelian groups.
Since $\C^\times$ is injective, the exact sequence splits.
\end{proof}
According to \cite{FLM}, we have the following:
\begin{prop}[{\cite[Proposition 5.2.3]{FLM}}]
\label{twist_free}
If $A$ is a free abelian group of finite rank,
then the map $c:H^2(A,\C^\times) \rightarrow Alt^2(A)$ is an isomorphism.
\end{prop}

{\it A twisted group algebra of an abelian group} $A$ is
an $A$-graded unital associative $\C$-algebra $R=\bigoplus_{a\in A}R_a$ satisfying the following conditions:
\begin{enumerate}
\item
$1 \in R_0$;
\item
$\dim_\C R_a =1$;
\item
$e_a e_b \neq 0$ for any $e_a \in R_a\setminus \{0\}$ and $e_b \in R_b\setminus \{0\}$.
\end{enumerate}
Twisted group algebras $R,S$ of $A$ are isomorphic if 
there is a $\C$-algebra isomorphism $h:R \rightarrow S$ which preserves the $A$-grading, that is, $h(R_a) = S_a$ for any $a \in A$.

Let $R$ be a twisted group algebra of $A$.
Let us choose a nonzero element $e_a$ of $R_a$ for all $a \in A$.
Then, define the map $f:A\times A \rightarrow \C^\times$
by $e_a e_b=f(a,b)e_{a+b}$.
Clearly, the cohomology class $f \in H^2(A,\C^\times)$ is independent of the choice of
$\{e_a \in R_a\}_{a\in A}$.
Furthermore, the cohomology class of $f$ uniquely determines the isomorphism class.
The alternative bilinear form $c_f \in Alt^2(A)$ is called {\it an associated commutator map 
of the twisted group algebra}, denoted by $c_R$.
Hence, we have:
\begin{lem}
\label{twisted_classify}
There exists a bijection between the isomorphism classes of twisted group algebras of an abelian group $A$
and $H^2(A,\C^\times)$.
\end{lem}

Let $L$ be an even lattice.
Define the alternative bilinear form on $L$ by $c_L(\al,\be)=(-1)^{(\al,\be)}$ for
$\al,\be \in L$.
By Proposition \ref{twist_free},
there exists a 2-cocycle $f \in Z^2(L,\C^\times)$
such that the associated commutator map $c_f$ is coincides with $c_L$.
The corresponding twisted group algebra is denoted by $\C\{L\}$,
which plays an important role in the construction of 
the lattice vertex algebra attached to $L$.

Hereafter, we generalize this construction to some abelian group $L$, called an even $H$- lattice, which is not always free.
Let $H$ be a finite-dimensional vector space over $\mathbb{C}$ equipped with a non-degenerate symmetric bilinear form $(-,-)$.
{\it An even $H$-lattice} is a subgroup $L \subset H$ such that
$(\al,\be) \in \Z$ and $(\al,\al)\in 2\Z$ for any $\al,\be \in L$,
denoted by $(L,H)$.
Define an alternative bilinear form on $L$ by
$$c_L(\al,\be)=(-1)^{(\al,\be)}$$
for $\al,\be\in L$.

Let $(L,H)$ be an even $H$-lattice.
Since the abelian group $L$ is not always free, we cannot directly apply Proposition \ref{twist_free} to the alternative bilinear form $c_L$.
We can, however, prove that there exists a 2-cocycle $f_L \in Z^2(L,\C^\times)$
such that the associated alternative bilinear form $c_{f_L}$ coincides with $c_L$. 
We construct $f_L$ by pulling back a cocycle on a quotient lattice of $L$, where the kernel of the quotient map is the radical of the bilinear form on $L$.

Let $E$ be the $\C$-vector subspace of $H$ spanned by $L$.
Set $E^\perp=\{v \in E \;|\; (v,E)=0 \}$.
Let $\pi: E \rightarrow E/E^\perp$ be the canonical projection
and $\bar{L}$ the image of $L$ under $\pi$.  
Then, the bilinear form on $H$ induces a bilinear form on $E/E^\perp$ and $\bar{L}$
and the bilinear form on $E/E^\perp$ is non-degenerate.
Since $\bar{L}$ is an even $E/E^\perp$-lattice and
spans $E/E^\perp$,
$\bar{L}$ is a free abelian group of rank $\dim_\C E/E^\perp$.
By Proposition \ref{twist_free}, there exits a 2-cocycle $g:\bar{L}\times \bar{L} \rightarrow \C^\times$
such that $g(\bar{a},\bar{b})g(\bar{b},\bar{a})^{-1}=(-1)^{(\bar{a},\bar{b})}$ for any
$\bar{a},\bar{b} \in \bar{L}$.
Then, the 2-cocycle $f_L$ defined by $f_L(a,b)=g(\pi(a),\pi(b))$ for $a,b \in L$
satisfies $c_{f_L}=c_L$.
The twisted group algebra constructed by the 2-cocycle $f_L \in Z^2(L,\C^\times)$
is denoted by $A_{L,H}$,
which is independent of the choice of the 2-cocycle $f_L \in Z^2(L,\C^\times)$.
Hence, we have:
\begin{prop}
\label{existence_twist}
For even $H$-lattice $(L,H)$, there exits a unique twisted group algebra $A_{L,H}$
such that the associated commutator map is $c_L$.
\end{prop}

\begin{rem}
By the construction, if $L$ is non-degenerate, then $A_{L,H}$ is isomorphic to
the twisted group algebra $\C\{L\}$.
\end{rem}

\subsection{AH pairs}\label{AH}
Let $H$ be a finite-dimensional vector space over $\mathbb{C}$ equipped with a non-degenerate symmetric bilinear form $(-,-)$ and $A$ an associative algebra over $\mathbb{C}$ with a non-zero unit element $1$. 
Assume that $A$ is graded by $H$ as $A=\bigoplus_{\alpha\in H}A^\alpha$. 

We will say that such a pair $(A,H)$ is an {\it AH pair} if the following conditions are satisfied:
{\leftmargini2.8em
\begin{enumerate}
\itemsep1ex\parskip0ex
\item[AH1)]
$1 \in A^0$ and $A^\alpha A^\beta\subset A^{\alpha+\beta}$ for any $\alpha,\beta\in H$ ;
\item[AH2)]
$A^\alpha A^\beta\ne 0$ implies $(\alpha,\beta)\in \mathbb{Z}$;
\item[AH3)]
For $v\in A^\alpha,w\in A^\beta$, $vw=(-1)^{(\alpha,\beta)}wv$. 
\end{enumerate}}

For AH pairs $(A,H_A)$ and $(B,H_B)$,
a homomorphism of AH pairs is a pair $(f,f')$ of maps $f:A \longrightarrow B$ and $f':H_A\longrightarrow H_B$ such that $f$ is an algebra homomorphism and $f'$ an isometry such that $f(A^\alpha)\subset B^{f'(\alpha)}$ for all $\alpha\in H_A$.
We denote by $\AH$ the category of AH pairs.
For an AH pair $(A,H)$, we set
\begin{align*}
M_{A,H}&=\{\alpha\in H\,|\,A^\alpha\ne 0\}.
\end{align*}

A {\it good AH pair} is an AH pair $(A,H)$ such that:
{\leftmargini2.5em
\begin{enumerate}\label{GAH}
\item[GAH1)]
$A^0$ is a one-dimensional vector space with the basis $\{1\}$;
\item[GAH2)]
$vw \neq 0$ for any $\al,\be \in M_{A,H}$, $v \in A^\al \setminus \{0\}$ and $w\in A^\be \setminus \{0\}$.
\end{enumerate}}\noindent
By isolating the leading terms with respect to some linear order on $H$, we have:
\begin{prop}
For a good AH pair $(A,H)$, $0$ is the only zero divisor in $A$.
\end{prop}
Let $(A,H)$ be a good AH pair.
Then, (GAH1) and (GAH2) directly imply $M_{A,H}$ is a submonoid of $H$.
Furthermore, (AH2) and (AH3) imply $(\al, \be) \in \Z$, $(\al,\al) \in 2\Z$ for any $\al,\be\in M_{A,H}$.

A {\it lattice pair} is a good AH pair $(A,H)$ such that:
\begin{enumerate}
\item[LP)]
$M_{A,H}$ is an abelian group, that is, $-\al\in M_{A,H}$ for any $\al \in M_{A,H}$.
\end{enumerate}

\begin{lem}
\label{lattice_pair_twisted}
If $(A,H)$ is a lattice pair,
then $A$ is a twisted group algebra of the abelian group $M_{A,H}$ whose commutator map is defined by $M_{A,H}\times M_{A,H} \rightarrow \C^\times, (\al,\be) \mapsto (-1)^{(\al,\be)}$.
\end{lem}
\begin{proof}
It suffices to show that $\dim A_\al=1$ for any $\al \in M_{A,H}$.
Let $\al \in M_{A,H}$.
By (LP), $A^{-\al} \neq 0$.
Let $a,a'$ be non-zero vectors in $A^\al$ and $b$ a non-zero vector in $A^{-\al}$.
By Definition (GAH1) and (GAH2),
$ab \in A^0=\C \1$ is non zero. We may assume that $ab=1$.
Then, $a'=(a\cdot b)\cdot a'=a\cdot (b\cdot a') \in \C a$. Hence, $A^\al=\C a$.
\end{proof}

By Lemma \ref{vanish} and Lemma \ref{twisted_classify}, Proposition \ref{existence_twist}, Lemma \ref{lattice_pair_twisted},
we have the following classification result of lattice pairs:
\begin{prop}
\label{lattice_pair_invariant}
For any even $H$-lattice $(L,H)$,
there exists a unique lattice pair $(A_{L,H},H)$ such that $M_{A_{L,H},H}=L$.
Furthermore, the lattice pairs $(A,H)$ and $(A',H')$ are isomorphic iff
there exists an isometry $f:H\rightarrow H'$ such
that $f(M_{A,H})=M_{A',H'}$.
\end{prop}

We end this subsection by defining a maximal lattice pair of a good AH pair.
Let $(A,H)$ be a good AH-pair.
For a submonoid $N \subset M_{A,H}$,
set $A_N=\bigoplus_{\al \in N}A^\al$.
Then, we have:
\begin{lem}
\label{submonoid}
For a submonoid $N \subset M_{A,H}$,
$(A_N,H)$ is a subalgebra of $(A,H)$ as an AH pair.
\end{lem}
Set 
\begin{align}
L_{A,H}=\{\alpha\in H\,|\,A^\alpha, A^{-\al}\ne 0\}=M_{A,H} \cap (-M_{A,H}) . \label{eq_lattice}
\end{align}
We denote $A_{L_{A,H}}$ by $A^{\lat}$.
Then, $(A^{\lat},H)$ is a lattice pair,
which is maximal in the following sense:
\begin{prop}
\label{lattice_universal}
Let $(B,H_B),(A,H_A)$ be good AH pairs and $(f,f'):(B,H_B) \rightarrow (A,H_A)$
 an AH pair homomorphism.
If $(B,H_B)$ is a lattice pair, then $f$ is injective and $f(B) \subset A^\lat$.
In particular, there is a natural bijection 
between $\Hom_{\text{AH pair}}((B,H_B),(A,H_A))$ and $\Hom_{\text{AH pair}}((B,H_B),(A^\lat,H_A))$.
\end{prop}

\begin{proof}
Since $f$ is an AH pair homomorphism, $\ker f$ is an $H_B$-graded ideal.
By the proof of Lemma \ref{lattice_pair_twisted}, every nonzero element in $B^{\al}$ is
invertible for any $\al \in M_{B,H_B}$. Hence, $\ker f=0$.
For $\al \in M_{B,H}$, we have $0 \neq f(B^\al) \subset A^{f'(\al)}$ and $0 \neq f(B^{-\al}) \subset A^{-f'(\al)}$. 
Thus, $f'(\al) \in L_{A,H_A}$ and $f(B) \subset A^\lat$.
\end{proof}

\subsection{Functor $\mathcal{V}$}\label{AH_vertex}
Let $H$ be a vector subspace of a vertex algebra $V$. 
We will say that $H$ is a {\it Heisenberg subspace}\/ of $V$ if the following conditions are satisfied:
{\leftmargini2.7em
\begin{enumerate}
\item[HS1)]
$h(1)h'\in \mathbb{C}\mathbf{1}$ for any $h,h'\in H$;
\item[HS2)]
$h(n)h'=0$ if $n\geq 2$ for any $h,h'\in H$;
\item[HS3)]
The bilinear form $(-,-)$ on $H$ defined by $h(1)h'=(h,h')\mathbf{1}$ for $h,h'\in H$ is non-degenerate.
\end{enumerate}}%
\noindent
We will call such a pair $(V,H)$ a {\it VH pair}. 
For VH pairs $(V,H_V)$ and $(W,H_W)$, a VH pair homomorphism from $(V,H_V)$ to $(W,H_W)$ is a vertex algebra homomorphism $f: V \rightarrow W$ such that $f(H_V)=H_W$. 
Since $f(h(1)h')=f(h)(1)f(h')$ for any $h,h' \in H_V$ and the bilinear form on $H_V$ is non-degenerate, $f |_{H_V}: H_V \rightarrow H_W$ is an isometric isomorphism of vector spaces. 
A vertex subalgebra $W$ of a VH pair $(V,H)$ is said to be a {\it subVH pair} if
$H$ is a subset of $W$.
$\VH$ is a category whose objects are VH pairs and morphisms are VH pair homomorphisms.

Let $(A,H)$ be an AH pair and consider the Heisenberg vertex algebra $M_H(0)$ associated with $(H,(-,-))$.
Let us identify the degree $1$ subspace of $M_H(0)$ with $H$. 
Then, the vertex algebra $M_H(0)$ is generated by $H$ and the actions of $h(n)$ for $h\in H$ and $n\in \mathbb{Z}$ give rise to an action of the Heisenberg Lie algebra $\widehat{H}$ on $M_H(0)$. 

Consider the vector space 
$$
V_{A,H}
=M_H(0)\otimes A
=\bigoplus_{\alpha\in H}M_H(0)\otimes A^{\alpha}.
$$
We let $\widehat{H}$ act on this space by setting, for $h\in H$, $v \in M_H(0)$ and $a\in A^\al$, 
$$
h(n)(v\otimes a)=
\begin{cases}
(h,\alpha) v\otimes a,&n=0,\cr
(h(n)v)\otimes a,&n\ne 0.
\end{cases}
$$
Let $\al \in H$ and $a\in A^\al$.
Denote by $l_a \in \End \;A$ the left multiplication by $a$
and define $ l_a z^\al : A \rightarrow A[z^\pm]$ by $l_a z^\al \cdot b= z^{(\al,\be)}ab$ for $b \in A^\be$,
where we used that $ab \neq 0$ implies $(\al,\be) \in \Z$.
Then, set 
$$
Y(\mathbf{1}\otimes a,z)=
\exp\biggl(\sum_{n\geq 1}\frac{\alpha(-n)}{n}z^{n}\biggr)\exp\biggl(\sum_{n\geq 1}\frac{\alpha(n)}{-n}z^{-n}\biggr)\otimes l_a z^{\alpha} \in \End V_{A,H}[[z^\pm]].
$$
The series $Y(h\otimes 1,z)=\sum_{n \in \Z} h(n) z^{-n-1}$ with $h\in H$ and $Y(1\otimes a,z)$ with $a\in A$ form a set of mutually local fields on $V_{A,H}$ and generate a structure of a vertex algebra on it.
Since $M_H(0)$ is a vertex subalgebra of $V_{A,H}$, $(V_{A,H},H)$ is canonically a VH pair.

In the case that $A=\mathbb{C}\{L\}$ is the twisted group algebra associated with an even lattice $L$ and $H=\mathbb{C}\otimes_{\mathbb{Z}}L$, the vertex algebra $V_{A,H}$ is nothing else but the lattice vertex algebra $V_L$.
\begin{prop}
\label{V_functor}
The above construction gives a functor from the category of AH pairs
to the category of VH pairs.
\end{prop}
We denote by $\mathcal{V}$ this functor,
thus, $\mathcal{V}(A,H)$ is a VH pair $(V_{A,H},H)$.

In order to prove the above proposition, we need the following result from \cite[Proposition 5.7.9]{LL}:
\begin{prop}[\cite{LL}]
\label{vertex_homomorphism}
Let $f$ be a linear map from a vertex algebra $V$ to a vertex algebra $W$ such
that $f(\1)=\1$ and such that
$$f(s(n)v)=f(s)(n)f(v) \text{ for any}\; s \in S, v \in V \text{ and}\; n \in \Z,$$
where $S$ is a given generating subset of $V$.
Then, $f$ is a vertex algebra homomorphism.
\end{prop}
\begin{proof}[proof of Proposition \ref{V_functor}]
Let $(A,H_A)$ and $(B,H_B)$ be AH pairs and 
$(f,f'):(A,H_A) \rightarrow (B,H_B)$ an AH pair homomorphism.
It is clear that there is a unique 
 $\C$-linear map $F: V_{A,H_A} \rightarrow V_{B,H_B}$
such that:
\begin{enumerate}
\item
$F(h(n)-)=f'(h)(n)F(-)$ for any $h \in H_A$;
\item
$F(\1\otimes a)=\1\otimes f(a)$ for any $a \in A$.
\end{enumerate}
It is easy to prove that the restriction of $F$ gives an isomorphism from $M_{H_A}(0)$ to
$M_{H_B}(0)$ and satisfies 
$F(v\otimes a)=F(v) \otimes f(a)$ for any $v\in M_{H_A}(0)$ and $a \in A$.
It suffices to show that $F$ is a vertex algebra homomorphism.
We will apply Proposition \ref{vertex_homomorphism} with $S=H_A \oplus A$.
Let $a\in A^\al$ and $b \in A^\be$ for $\al,\be \in H_A$ and $v \in M_{H_A}(0)$.
Since 
\begin{align*}
F(Y(\1\otimes a,z) v\otimes b) \\
&=z^{(\alpha,\beta)}
F(\exp\biggl(-\sum_{n\leq -1}\frac{\alpha(n)}{n}z^{-n}\biggr)\exp\biggl(-\sum_{n\geq 1}\frac{\alpha(n)}{n}z^{-n}\biggr) v\otimes ab) \\
&=z^{(f'(\al),f'(\be))}\exp\biggl(-\sum_{n\leq -1}\frac{f'(\alpha)(n)}{n}z^{-n}\biggr)\exp\biggl(-\sum_{n\geq 1}\frac{f'(\alpha)(n)}{n}z^{-n}\biggr)
f(a)F(v\otimes b) \\
&= Y(\1\otimes f(a),z)F(v\otimes b),
\end{align*}
$F$ is a vertex algebra homomorphism.
\end{proof}

For an even $H$-lattice $(L,H)$, the VH pair $V_{{A_{L,H},H}}$ is
called {\it $H$-lattice vertex algebra}.
We denote it by $V_{L,H}$.

\begin{rem}
\label{lattice_functor}
The construction of the lattice vertex algebra $V_L$ from an even lattice $L$ does not form a functor,
since there is no natural homomorphism from the automorphism group of the lattice $L$ to the automorphism group of the twisted group algebra $\C\{L\}$ (The automorphism group of the twisted group algebra is an extension of the automorphism group of the lattice, but it is not necessarily split).
This is one reason we introduce AH pairs.
\end{rem}

\section{VH pairs and Associative Algebras}\label{sec_VH}
In the previous section, we constructed a functor 
$\mathcal{V}:\AH \rightarrow \VH$.
In this section, we construct a right adjoint functor $\Omega:\VH \rightarrow \AH$ of $\mathcal{V}$ (see Theorem \ref{adjoint}).
That is, we construct a `universal' associative algebra from a VH pair. Section \ref{sec_vacuum_algebra} is devoted to constructing the functor $\Omega$,
while in \ref{sec_adjoint}, we prove that $\Omega$ and $\mathcal{V}$ form an adjoint pair.
In subsection \ref{sec_good_VH}, we combine the results in this section and previous section and classify $H$-lattice vertex subalgebras of a VH pair.
%
\subsection{The functor $\Omega$}\label{sec_vacuum_algebra}
Let $(V,H)$ be a VH pair and $\omega_{H}$ be the canonical conformal vector of $M_H(0)$ given by the Sugawara construction,
that is, $\omega_H=
\frac{1}{2}\sum_i h_i(-1)h^i$, where $\{h_i\}_i$ is a basis of $H$ and $\{h^i\}_i$ is the dual basis of $\{h_i\}_i$ with respect to the bilinear form on $H$.
For $\alpha\in H$, we let $\Omega_{V,H}^\alpha$ be the set of all vectors $v\in V$ such that:
\begin{enumerate}
\item[VS1)] $T_Vv=\omega_H(0)v$.
\item[VS2)] $h(n)v=0$ for any $h\in H$ and $n\geq 1$.
\item[VS3)] $h(0)v=(h,\alpha)v$ for any $h\in H$. 
\end{enumerate}
Here, $T_V$ is the canonical derivation of $V$ defined by $T_V v=v(-2)\mathbf{1}$.
Set 
$$
\Omega_{V,H}=\bigoplus_{\alpha\in H}\Omega_{V,H}^{\alpha}.
$$
A vector of a Heisenberg module which satisfies the condition (VS2) and (VS3)
is called an {\it $H$-vacuum vector}.
\begin{rem}
\label{rem_ker}
Since $\omega_H(0)-T_V$ is a derivation on $V$,
$\ker(\omega_H(0)-T_V)$ is a vertex subalgebra of $V$.
Furthermore, since $H \subset \ker(\omega_H(0)-T_V)$,
$(\ker(\omega_H(0)-T_V),H)$ is a VH pair
and the map $(V,H) \mapsto (\ker(\omega_H(0)-T_V),H)$ defines 
a functor from $\underline{\text{VH pair}}$ to itself.
\end{rem}

\begin{lem}
\label{vac_omega}
For $v \in \Omega_{V,H}^\al$, $\omega_{H}(0)v=\al(-1)v$ and $\omega_{H}(1)v=\frac{(\al,\al)}{2}v$. Furthermore, $\Omega_{V,H}^0=\ker T_V$.
\end{lem}

\begin{proof}
Let $\{h_i\}_i$ be a basis of $H$ and $\{h^i\}_i$ the dual basis with respect to the bilinear form on $H$. Then,
$\omega_{H}(0)v=\frac{1}{2} \sum_i (h_i(-1)h^i)(0)v=
\frac{1}{2}\sum_i \sum_{n \geq 0} (h_i(-1-n)h^i(n) + h^i(-n-1)h_i(n))v
=\frac{1}{2} \sum_i (h^i,\al) h_i(-1)v+(h_i,\al) h^i(-1)v=\al(-1)v.$
Similarly,
$\omega_{H}(1)v=\frac{1}{2} \sum_i (h_i(-1)h^i)(1)v=
\frac{1}{2}\sum_i \sum_{n \geq 0} (h_i(-1-n)h^i(n+1) + h^i(-n)h_i(n))v
=\frac{1}{2} \sum_i (h^i,\al)(h_i, \al) v=\frac{(\al, \al)}{2}v.$
If $v \in \Omega_{V,H}^0$, then $T_V v=\omega_{H}(0)v=0$.
If $v \in \ker T_V$, then $a(n)v=0$ for any $n \geq 0$ and $a \in V$, which implies $v \in \Omega_{V,H}^0$.
Thus, $\Omega_{V,H}^0=\ker T_V$.
\end{proof}

The following simple observation is fundamental:
\begin{lem}
\label{product}
Let  $v \in \Omega_{V,H}^\al$ and $w \in \Omega_{V,H}^\beta$.
If $Y(v,z)w \neq 0$, then
$(\al,\be) \in \Z$ and $v(-(\al,\be)+i)w =0$ for any $i \geq 0$ and
 $0 \neq v(-(\al,\be)-1)w=(-1)^{(\al,\be)}w(-(\al, \be)-1)v \in \Omega_{V,H}^{\al+\be}$. In particular, if 
$(\al,\be) \notin \Z$, then $Y(v,z)w=0$.
\end{lem}

\begin{proof}
Let $k$ be an integer such that $v(k)w \neq 0$ and $v(k+i)w=0$ for any $i >0$.
We claim that $v(k)w \in \Omega_V^{\al+\beta}$.
Since $[h(n),v(k)]=\sum_{i \geq 0}\binom{n}{i}(h(i)v)(n+k-i)=(h,\al)v(n+k)$,
we have $h(n)v(k)w=0$ and $h(0)v(k)w=(h,\al+\be)v(k)w$ for any $h \in H$ and $n \geq 1$.
Furthermore, $\omega_{H}(0)v(k)w=(\omega_{H}(0)v)(k)w+v(k)\omega_{H}(0)w=(T_Vv)(k)w+v(k)T_Vw=T_V(v(k)w)$, which implies $v(k)w \in \Omega_{V,H}^{\al+\be}$.
By Lemma \ref{vac_omega}, $\omega_{H}(1)v(k)w = \frac{(\al+\be,\al+\be)}{2}v(k)w$. 
Since $\omega_{H}(1)v(k)w=[\omega_{H}(1), v(k)]w + v(k)\omega_{H}(1)w
=\sum_{i \geq 0} \binom{1}{i}(\omega_{H}(i)v)(k+1-i)w+v(k)\omega_{H}(1)w
=((\al,\al)/2+(\be,\be)/2)v(k)w+ (T_Vv)(k+1)w=
((\al,\al)/2+(\be,\be)/2-k-1)v(k)w$,
we have $k=-(\al,\be)-1$. Hence, $(\al,\be) \in \Z$.
By applying skew-symmetry, $v(-(\al,\be)-1)w=
\sum_{i \geq 0}(-1)^{i+(\al,\be)} T_V^i/i! w(-(\al,\be)-1+i)v=(-1)^{(\al,\be)}w(-(\al,\be)-1)v$.
\end{proof}

We may equip $\Omega_{V,H}$ with a structure of an associative algebra in the following way: For $v\in \Omega_{V,H}^{\alpha}$ and $w\in \Omega_{V,H}^{\beta}$, define the product $vw$ by
$$
vw=\begin{cases}
v(-(\alpha,\beta)-1)w,&\hbox{if $Y(v,z)w\ne 0$},\cr
\hfil 0,&\hbox{otherwise}.\cr
\end{cases}
$$

In order to show that the product is associative,
we need the following lemma.
\begin{lem}
\label{associative}
Let $a,b,c \in V$, $p,q,r \in \Z$ satisfy $a(r+i)b=0$, $b(q+i)c=0$ and $a(p-1+i)c=0$ for any $i \geq 1$. Then, $(a(r)b)(p+q)c=a(r+p)(b(q)c)$.
\end{lem}

\begin{proof}
Applying the Borcherds identity, we have
$(a(r)b)(p+q))c=
\sum_{i \geq 0} \binom{p}{i} (a(r+i)b)(p+q-i)c
=\sum_{i \geq 0} (-1)^i \binom{r}{i}(a(r+p-i)b(q+i)c-(-1)^r b(r+q-i)a(p+i)c)
= a(r+p)b(q)c$.
\end{proof}

\begin{lem}
$(\Omega_{V,H},H)$ is an AH pair.
\end{lem}

\begin{proof}
Let $a \in \Omega_{V,H}^\al$, $b \in \Omega_{V,H}^\be$ and $c \in \Omega_{V,H}^\gamma$ for $\al, \be,\gamma \in H$. 
By Lemma \ref{product}, it remains to show that the product is associative, i.e., $(ab)c=a(bc)$.
Suppose that one of $(\al,\be),(\be,\gamma),(\al,\gamma)$ is not integer.
By Lemma \ref{product}, it is easy to show that $(ab)c=a(bc)=0$.
Thus, we may assume that $(\al,\be),(\be,\gamma),(\al,\gamma) \in \Z$.
Then, $(ab)c=(a(-(\al,\be)-1)b)(-(\al+\be,\gamma)-1)c$ and $a(bc)=a(-(\al,\be+\gamma)-1)(b(-(\be,\gamma)-1)c)$.
By applying Lemma \ref{associative} with $(p,q,r)=(-(\al,\gamma), -(\be,\gamma)-1, -(\al,\be)-1)$, we have $(ab)c=a(bc)$.
\end{proof}

\begin{prop}
\label{func_omega}
The map $\Omega: \underline{\text{VH pair}  }  \rightarrow \underline{\text{AH pair}}
,\; (V,H) \mapsto (\Omega_{V,H},H)$ is a functor.
\end{prop}

\begin{proof}
Let $(V,H_V)$ and $(W,H_W)$ be VH pairs and $f:(V,H_V) \rightarrow (W,H_W)$ a VH pair homomorphism.
Set $f'= f|_{H_V} : H_V \rightarrow H_W$.
Let $e_\al \in \Omega_{V,H_V}^{\al}$.
For $h \in H_W$, since $h(n)f(e_\al)=f(f'^{-1}(h))(n)f(e_\al)
=f(f'^{-1}(h)(n)e_\al)$, we have $h(n)f(e_\al)=0$ and 
$h(0)f(e_\al)= (f'^{-1}(h), \al)f(e_\al)$ for any $n \geq 1$.
Since $f'$ is an orthogonal transformation, $(f'^{-1}(h), \al)=(h, f'(\al))$.
Since $f(\omega_{H_V})=\omega_{H_W}$,
$\omega_{H_W}(0)f(e_\al)=f(\omega_{H_V}(0)e_\al)=f(e_\al(-2)\1)=f(e_\al)(-2)\1$.
Hence, $f(\Omega_V^{\al}) \subset \Omega_W^{f'(\al)}$.
The rest is obvious.
\end{proof}
We will often abbreviate the AH pair (and the functor) $(\Omega_{V,H},H)$ as $\Omega(V,H)$.

Since $\Omega$ is a functor, we have:
\begin{lem}
\label{omega_auto}
Let $f \in \Aut V$ such that $f(H) =H$. Then, $f$ induces an automorphism of the AH pair $\Omega(V,{H})$.
\end{lem}

Let $(V,H_V)$ and $(W,H_W)$ be VH pairs.
Then, $V\otimes W$ is a vertex algebra, and $H_V \oplus H_W$ is a Heisenberg subspace of $V \otimes W$.
Hence, we have:
\begin{lem}
\label{VH_monoidal}
The category of VH pairs and the category of AH pairs are strict symmetric monoidal categories.
\end{lem}

It is easy to confirm that the functor $\mathcal{V}: \underline{\text{AH pair}} \rightarrow \underline{\text{VH pair}}$ is a monoidal functor, that is, $\mathcal{V}((A,H_A)\otimes (B,H_B))$ is naturally isomorphic to $\mathcal{V}(A,H_A)\otimes \mathcal{V}(B,H_B)$.
However,  $\Omega$ is not strict monoidal functor,
since $\ker(\omega_H(0)-T)$ does not preserve the tensor product (see Remark \ref{rem_ker}).

\begin{lem}
\label{omega_tensor}
Let $(V,H_V)$ and $(W,H_W)$ be VH pairs.
If $\omega_{H_W}(0)=T_W$,
then $\Omega(V \otimes W,H_V \oplus H_W)$ is isomorphic to 
$\Omega(V,H_V) \otimes \Omega(W,H_W)$ as an AH-pair.
\end{lem}

\begin{proof}
It is clear that $\Omega_{V,H_V} \otimes \Omega_{W,H_W} \subset \Omega_{V\otimes W,H_V\oplus H_W}$.
Let $\sum_i v_i \otimes w_i \in \Omega_{V\otimes W,H_V\oplus H_W}$.
We can assume that $\{v_i\}_i$ and $\{w_i\}_i$ are both linearly independent.
For $h \in H_V$ and $n \geq 1$, $h(n)\sum_i v_i\otimes w_i=
\sum_i h(n)v_i \otimes w_i=0$. Since $w_i$ are linearly independent, $h(n)v_i=0$ for any $i$. Similarly, $h'(n)w_i=0$ for any $i$ and $h' \in H_W$.
Since $\omega_{H_V \oplus H_W}=\omega_{H_V} \otimes \1 + \1\otimes \omega_{H_W}$, we have
$\omega_{H_V}(0) \sum_i v_i\otimes w_i =
(\omega_{H_V \oplus H_W}(0)-\omega_{H_W}(0)) \sum_i v_i\otimes w_i
=(T_{V\otimes W}-1\otimes T_W) \sum_i v_i\otimes w_i
=(T_V \otimes 1) \sum_i v_i\otimes w_i$, where we used $\omega_{H_W}(0)=T_W$ and 
$\sum_i v_i\otimes w_i \in \Omega_{V\otimes W,H_V\oplus H_W}$.
Hence, $\omega_{H_V}(0)v_i=T_V v_i$, which implies that $v_i \in \Omega_{V,H_V}, w_i \in \Omega_{W,H_W}$.
\end{proof}

%


\subsection{Adjoint functor}\label{sec_adjoint}
In this subsection, we will prove the following theorem:
\begin{thm}
\label{adjoint}
The functor $\Omega: \VH  \rightarrow \AH$ is a right adjoint to the functor $\mathcal{V}: \AH \rightarrow \VH$.
\end{thm}

The following observation is important:
\begin{lem}
\label{AH_ker}
Let $(A,H)$ be an AH pair. Then, the following conditions hold for the VH pair $(V_{A,H},H)$:
\begin{enumerate}
\item
$T_{V_{A,H}}=\omega_H(0)$;
\item
The AH pair $\Omega({V_{A,H},H})$ is isomorphic to $(A,H)$;
\item
The vertex algebra $V_{A,H}$ is generated by the subspaces $H$ and $\Omega_{V_{A,H},H}=A$ as
a vertex algebra.
\end{enumerate}
In particular, the composite functor $\Omega \circ \mathcal{V}: \AH \rightarrow \AH$ is
 isomorphic to the identity functor.
\end{lem}

Hereafter, we will construct a natural transformation  
from $\mathcal{V} \circ \Omega: \underline{\text{VH pair}} \rightarrow \underline{\text{VH pair}}$ to the identity functor and prove the above Theorem.
We start from the observations in Lemma \ref{AH_ker}.
Let $(V,H_V)$ be a VH pair
and $(A,H_A)$ an AH pair and $f:\mathcal{V}(A,H_A) \rightarrow (V,H_V)$ a VH pair homomorphism.
Then, by Lemma \ref{AH_ker} and $f(\omega_{H_A})=\omega_{H_V}$, 
the image of $f$ satisfies the condition (VS1) and is generated by $H_V$ and the subset of $H_V$-vacuum vectors as a vertex algebra.
So, let $W$ be the vertex subalgebra of $V$ generated by $H_V$ and $\Omega_{V,H_V}$.
Then, there is a natural isomorphism
 $$\Hom_{\VH}((\mathcal{V}_{A,H_A},H_A),(V,H_V)) \rightarrow
\Hom_{\VH}( (\mathcal{V}_{A,H_A},H_A), (W,H_V)).$$
Hence, the proof of Theorem \ref{adjoint} is divided into two parts.
First, we prove that the subalgebra $W$ is isomorphic to
$V_{\Omega_{V,H_V},H_V}=\mathcal{V}\circ \Omega (V,H_V)$ as a VH pair,
which gives us a natural transformation $\mathcal{V}\circ \Omega (V,H_V) \hookrightarrow (V,H_V)$.
Second, we prove that there is a natural isomorphism
$$\Hom_{\AH}((A,H_A),(B,H_B)) \rightarrow \Hom_{\VH}(\mathcal{V}(A,H_A),
\mathcal{V}({B,H_B})),$$
for AH pairs $(A,H_A)$ and $(B,H_B)$.

Let us prove the first step.
Since $M_{H_V}(0)$ is a subalgebra of $W$, $W$ is an $M_{H_V}(0)$-module.
\begin{lem}
\label{str_w}
The subspace $W \subset V$ is isomorphic to $M_{H_V}(0)\otimes_{\mathbb{C}} \Omega_{V,H_V}$
as an $M_{H_V}(0)$-module.
\end{lem}
\begin{proof}
Let $W'$ be an $M_{H_V}(0)$-submodule of $V$ generated by $\Omega_{V,H_V}$.
Then, according to the representation theory of Heisenberg Lie algebras (see \cite[Theorem 1.7.3]{FLM}), $W' \cong M_{H_V}(0) \otimes_{\mathbb{C}} \Omega_{V,H_V}$ as an $M_{H_V}(0)$-module.
Hence, it suffices to show that $W'$ is closed under the products of the vertex algebra.
Let $v,w \in W'$. We may assume that $v=h_1(i_1)\ldots h_k(i_k)e_\al$ and 
$w=h'_1(j_1)\ldots h'_l(j_l)e_\be$ where $h_i, h'_j \in H_V$ and $e_\al \in \Omega_{V,H_V}^\al$, $e_\be \in \Omega_{V,H_V}^\be$.
Since $(h(n)v)(m)w=\sum_{i \geq 0} (-1)^i \binom{n}{i} h(n-i)v(m+i)w-(-1)^n v(n+m-i)h(i)w$,
we may assume that $v=e_\al \in \Omega_{V,H_V}^\al$.
Since $e_\al(n)h(m)w=[e_\al(n), h(m)]w + h(m)e_\al(n)w
=-(h,\al)e_\al(n+m)w + h(m)e_\al(n)w$, we may assume that $w = e_\be \in \Omega_{V,H_V}^\be$.
If $Y(e_\al,z)e_\be=0$, then there is nothing to prove. Assume that  $Y(e_\al,z)e_\be \neq 0$.
Then, by Lemma \ref{product}, $e_\al(-(\al,\be)+k)e_\be =0$ for any $k \geq 0$ and 
$e_\al(-(\al,\be)-1)e_\be \in \Omega_V^{\al+\be}$.
We will show that $e_\al (-(\al,\be)-k)e_\be \in W'$ for $k \geq 2$.
By Lemma \ref{vac_omega}, $\omega_{H_V}(0)e_\be=\be(-1)e_\be$.
Then, $\be(-1)e_\al(-(\al,\be)-k)e_\be=[\be(-1), e_\al(-(\al,\be)-k)]e_\be+e_\al(-(\al,\be)-k)\be(-1)e_\be
=(\al,\be)e_\al(-(\al,\be)-k-1)e_\be + e_\al(-(\al,\be)-k)\omega_{H_V}(0)e_\be
=\omega_H(0)e_\al(-(\al,\be)-k)e_\be -k e_\al(-(\al,\be)-k-1)e_\be$.
Since $k \geq 1$, we have 
\begin{align}
e_\al(-(\al,\be)-k-1)e_\be &= \frac{1}{k}(\omega_{H_V}(0)-\be(-1))e_\al(-(\al,\be)-k)e_\be \label{lattice_identity} \\
&=\frac{1}{k!}(\omega_{H_V}(0)-\be(-1))^k e_\al(-(\al,\be)-1)e_\be.\nonumber 
\end{align}
Since $\omega_{H_V}(0)W' \subset W'$, we have $e_\al (-(\al,\be)-k-1)e_\be \in W'$.
%
\end{proof}

\begin{rem}
Equation (\ref{lattice_identity}) was proved by \cite{Roi} for lattice vertex algebras.
\end{rem}

\begin{lem}
\label{sub_Cartan}
The subalgebra $W$ is isomorphic to 
$\mathcal{V}\circ \Omega (V,H_V)$ as a VH pair. 
\end{lem}

\begin{proof}
By Lemma \ref{str_w}, we have an isomorphism 
$f: W \rightarrow V_{\Omega_{V,H_V},H_V}= M_{H_V}(0) \otimes \Omega_{V,H_V}$
as an $M_{H_V}(0)$-module.
We will apply Proposition \ref{vertex_homomorphism} with $S=H_V \oplus \Omega_{V,H_V}$.
It suffices to show that $f(a(n)v)=f(a)(n)f(v)$ for any $a \in \Omega_{V,H_V}$ and $v \in W$ and $n \in \Z$, which can be proved by the similar arguments as Lemma \ref{str_w}.
\end{proof}

As discussed above, the subalgebra $V_{\Omega_{V,H_V},H_V} \subset V$ is universal as follows:
\begin{prop}
\label{omega_factor}
Let $(A,H_A)$ be an AH pair and $f: (V_{A,H_A},H_A) \rightarrow (V,H_V)$ be a VH pair homomorphism.
Then, 
 $f(V_{A,H_A}) \subset V_{\Omega_{V,H_V},H_V}$.
In particular, there is a natural bijection,
$$\mathrm{Hom}_{\VH}((\mathcal{V}({A,H_A}), (V,H_V))\cong \mathrm{Hom}_{\VH}(\mathcal{V}(A,H_A), 
\mathcal{V}\circ\Omega(V,H_V)).$$
\end{prop}

Let us prove the second claim.
Let $(B,H_B)$ be an AH pair.
Since $\mathcal{V}$ and $\Omega$ are functors, we have natural maps
$$\phi:\mathrm{Hom}_{\AH}((A,H_A),(B,H_B)) \rightarrow \mathrm{Hom}_{\VH}(\mathcal{V}(A,H_A),\mathcal{V}(B,H_B))$$
and
$$\psi:
\mathrm{Hom}_{\VH}(\mathcal{V}(A,H_A),\mathcal{V}(B,H_B)) \rightarrow 
\mathrm{Hom}_{\AH}(\Omega \circ \mathcal{V}(A,H_A),\Omega \circ \mathcal{V}(B,H_B)).
$$

By Lemma \ref{AH_ker}, 
$$\mathrm{Hom}_{\AH}(\Omega \circ \mathcal{V}(A,H_A),\Omega \circ \mathcal{V}(B,H_B)) \cong
\mathrm{Hom}_{\AH}((A,H_A),(B,H_B)).
$$
It is easy to prove that $\psi$ is an inverse of $\phi$.
Hence, we have:
\begin{lem}
\label{omega_bijection}
The map $\phi$
 is a bijection.
\end{lem}
Theorem \ref{adjoint} follows immediately from the combination of Proposition \ref{omega_factor} and Lemma \ref{omega_bijection}.

\subsection{Good VH pairs and Maximal lattices}\label{sec_good_VH}
We will say that a VH pair $(V,H)$ is a {\it good VH pair}\/
if the AH pair $\Omega(V,H)$ is a good AH pair,
i.e., the following conditions are satisfied:
{\leftmargini1.5em
\begin{enumerate}
\item[1)]
$\Omega_{V,H}^0=\mathbb{C}\1$.
\item[2)]
$vw \neq 0$ for all $\alpha,\beta\in H$ and $v \in \Omega_{V,H}^{\alpha}\setminus \{0\}$ and
$w \in \Omega_{V,H}^{\beta}\setminus \{0\}$. 
\end{enumerate}}

The following proposition follows from \cite[Proposition 11.9]{DL}:
\begin{prop}
\label{simple_vertex}
If $V$ is a simple vertex algebra, then for any nonzero elements $a,b \in V$, $Y(a,z)b \neq 0$, i.e., there exists an integer $n \in \Z$ such that $a(n)b \neq 0$.
\end{prop}

By the above proposition and Lemma \ref{product} and Lemma \ref{vac_omega}, we have the following lemma:
\begin{lem}
\label{omega_simple}
For a VH pair $(V,H_V)$, if $V$ is simple and $\mathrm{Ker}\,T_V = \C\1$, then 
$(V,H_V)$ is a good VH pair.
\end{lem}

The kernel of $T_V$ is called the center of $V$ and it is well-known that $\mathrm{Ker}\,T_V = \C\1$
if $V$ is simple and has countable dimension over $\C$ (see for example, \cite[Section 6]{Ma}).

Hereafter, we will study good VH pairs.
\begin{lem}
\label{tensor_lattice}
For an even H-lattice $(L,H)$ and a good VH pair $(V,H_V)$,
$(V \otimes V_{L,H}, H_V \oplus H)$ is a good VH pair.
Furthermore, $\Omega(V \otimes V_{L,H}, H_V \oplus H)$ is isomorphic to
$\Omega(V,H_V) \otimes (A_{L,H},H)$ as an AH pair.
\end{lem}
\begin{proof}
For an $H$-lattice vertex algebra $V_{L,H}$, we have $\omega_{H}(0)=T_{V_{L,H}}$.
Hence, by Lemma \ref{omega_tensor},
we have $\Omega_{V\otimes V_{L,H}} \cong \Omega_V \otimes \Omega_{V_{L,H}}$.
Since $V$ and $V_{L,H}$ are good VH pairs,
$\Omega_{V\otimes V_{L,H}}^0
=\Omega_V^0 \otimes \Omega_{V_{L,H}}^0=\C\1\otimes \1$.
Let $v \in \Omega_{V\otimes V_{L,H}}^{\al+\be}$ and 
$v' \in \Omega_{V\otimes V_{L,H}}^{\al'+\be'}$ be nonzero elements, where $\al,\al' \in H_V$ and $\be, \be' \in H$.
Since $V_{L,H}$ is an $H$-lattice vertex algebra, 
$\dim \Omega_{V_{L,H}}^{\be}=\dim A_{L,H}^{\be}=1$ for any $\beta \in L$.
Hence, we may assume that $v=a \otimes b$ and $v'=a'\otimes b'$, where
$a \in \Omega_V^\al$ and $a' \in \Omega_V^{\al'}$, 
$b \in A_{L,H}^\be$, $b' \in A_{L,H}^{\be'}$.
Then, $vw=aa'\otimes bb' \neq 0$.
Hence, $(V \otimes V_{L,H}, H_V \oplus H)$ is a good VH pair.
\end{proof}

The category of good VH pairs (resp. good AH pairs) are full subcategory of 
$\VH$ (resp. $\AH$) whose objects are good VH pairs (reps. good AH pairs).
The following proposition follows from the Theorem \ref{adjoint}:
\begin{cor}
The adjoint functors in Theorem \ref{adjoint} induce adjoint functors between
 $\underline{\text{good VH pair}}$ and $\underline{\text{good AH pair}}$.
\end{cor}

Let $(V,H_V)$ be a good VH pair.
Then, by Proposition \ref{lattice_universal}, $\Omega(V,H_V)^\lat$ is a lattice pair.
Denote the even H-lattice $(L_{\Omega(V,H_V)},H_V)$ by $L_{V,H_V}$, which we call a {\it maximal lattice of the
good VH pair} $(V,H_V)$ (see Definition \ref{eq_lattice}).
By Proposition \ref{lattice_pair_invariant}, the AH pair $\Omega(V,H_V)^\lat$ is isomorphic to 
the AH pair $(A_{L_{V,H_V}, H_V},H_V)$.
By Lemma \ref{sub_Cartan}, the $H$-lattice vertex algebra
$V_{L_{V,H_V},H_V}=\mathcal{V}(\Omega(V,H_V)^\lat)$ is a subVH pair of $(V,H_V)$,
which is a maximal $H$-lattice vertex algebra as follows:
\begin{thm}
\label{maximal_lattice}
Let $(V,H)$ be a good VH pair.
Then, the subVH pair $\mathcal{V}(\Omega(V,H_V)^\lat)$ can be characterized by the following properties:
\begin{enumerate}
\item
$\mathcal{V}(\Omega(V,H_V)^\lat)$ is an $H$-lattice vertex algebra.
\item
For any even H-lattice $(L,H)$ and any VH pair homomorphism $f:(V_{L,H},H)\rightarrow (V,H_V)$, the image of $f$ is in
$\mathcal{V}(\Omega(V,H_V)^\lat)$.
\end{enumerate}
Furthermore, there is a natural bijection between
$\Hom_{\VH}((V_{L,H},H),(V,H_V))$ and \\
$\Hom_{\AH}((A_{L,H},H), (A_{L_{V,H_V},H_V})).$
\end{thm}
\begin{proof}
By Theorem \ref{adjoint} and Proposition \ref{lattice_universal},
\begin{align*}
\Hom_{\VH}((V_{L,H},H),(V,H_V))&\cong \Hom_{\underline{\text{good AH pair}}}((A_{L,H},H), \Omega(V,H_V)) \\
&\cong \Hom_{\underline{\text{good AH pair}}}((A_{L,H},H), \Omega(V,H_V)^\lat) \\
&\cong \Hom_{\VH}((V_{L,H},H), \mathcal{V}(\Omega(V,H_V)^\lat)).
\end{align*}
\end{proof}
By Theorem \ref{maximal_lattice} and Lemma \ref{tensor_lattice},
we have:
\begin{cor}
\label{tensor_max}
For a good VH pair $(V,H_V)$ and an even H-lattice $(L,H)$,
the maximal lattice of the good VH pair $(V\otimes V_{L,H}, H_V\oplus H)$ is
$(L_{V,H_V}\oplus L, H_V\oplus H)$.
\end{cor}
We end this section by showing the existence of lattice vertex algebras in a good VH pair.
Let $(V,H)$ be a good VH pair and $(L_{V,H},H)$ be the maximal lattice.

\begin{lem}
\label{existence}
If a subgroup $M \subset L_{V,H}$ spans a non-degenerate subspace of $H$,
then there exists a vertex subalgebra of $V$ which is isomorphic to the lattice vertex algebra $V_M$.
\end{lem}

\begin{proof}
It is clear that $M$ is a free abelian group of finite rank.
By Lemma \ref{submonoid}, $(\Omega(V,H)^\lat)_M \cong (A_{L_{V,H},H})_M$ is isomorphic to
the twisted group algebra $\C\{M\}$ associated with the even lattice $M$.
Set $H'=\{h \in H\;|\; (h,\al)=0 \fora \al \in M \}$.
Then, the vertex algebra $V_{\C\{M\},H}$ is isomorphic to $M_{H'}(0)\otimes V_M$ as a VH pair.
Thus, $V_M$ is a subalgebra of $V$.
\end{proof}

\section{Genera of VH pairs and Mass Formula} \label{sec_genus}
In this section, we introduce the notion of a \emph{genus} of VH pairs and prove a Mass formula (Theorem \ref{mass_vertex}) which is an analogous result of that for lattices.

Recall that two lattices $L_1$ and $L_2$ are said to be equivalent or in the same genus if their base changes are isomorphic as lattices:
$$L_1\otimes_{\Z}\mathbb{R}\simeq L_2\otimes_{\Z}\mathbb{R},\quad L_1\otimes_{\Z}\Z_p\simeq L_2\otimes_{\Z}\Z_p,$$
for all the prime integers $p$. 
Consider the unique even unimodular lattice $\tw$ of signature $(1,1)$. 
The proof of the following lemma can be found in \cite{KP}:
\begin{lem}
\label{def_genus}
The lattices $L_1$ and $L_2$ are in the same genus if and only if 
$$L_1 \otimes \tw\simeq L_2 \otimes \tw$$ 
as lattices.
\end{lem}

We will define an equivalence relation on VH pairs that is analogous to the equivalence relation on lattices given in this lemma.
For the rest of this paper, we always assume VH pairs to be good and denote a 
VH pair by $V$ instead of $(V,H_V)$ and the associated AH pair
 by $\Omega_V$ instead of $(\Omega_V^{H_V},H_V)$ for simplicity.
Note that tensor products of good VH pairs and $H$-lattice vertex algebras
are again good VH pairs by Lemma \ref{tensor_lattice}.

Let $V_\tw$ denote the lattice vertex algebra associated with the lattice $\tw$
and consider it as a VH pair $(V_\tw,H_\tw)$ with $H_\tw=\C\otimes_\Z \tw$.
Two (good)  VH pairs $V$ and $W$ are said to be \emph{equivalent} if there exists an isomorphism
$$f:\ V\otimes V_\tw\simeq W\otimes V_\tw$$
of VH pairs, that is, a vertex algebra isomorphism $f$ which satisfies 
$$f(H_{V\otimes V_\tw})=H_{V\otimes V_\tw}.$$
We call an equivalence class of VH pairs, a \emph{genus} of VH pairs, and denote by $\gen(V)$ the equivalence class of the VH pair $V$.

Before going into details, we briefly describe how to use the methods developed in the previous section to study the genera.
Assume that $V\otimes V_\tw \cong W\otimes V_\tw$ as a VH pair,
 that is, they are in the same genus of VH pairs.
We observe that $V$ is obtained as a coset of $V\otimes V_\tw$ by $V_\tw$,
that is, $V \cong \mathrm{Coset}_{V\otimes V_\tw}(V_\tw)$, and so is $W$ (see Section \ref{sec_coset}).
Furthermore, it is clear that the subVH pair $(M_{H_V}(0)\otimes V_\tw,H_V\oplus H_\tw) \subset (V\otimes V_\tw,H_V\oplus H_\tw)$ is an $H$-lattice vertex algebra,
whose maximal lattice is $(\tw, H_V\oplus H_\tw)$.
Conversely, in Section \ref{sec_coset}, we will show that
for each subVH pair of $V\otimes V_\tw$ which is isomorphic to the $H$-lattice vertex algebra
$M_{H_V}(0)\otimes V_\tw$,
we can obtain a VH pair $W$ such that $W\otimes V_\tw \cong V \otimes V_\tw$ by using the coset construction.
Thus, in order to determine $\gen(V)$, it suffices to classify the subVH pairs which are isomorphic to the $H$-lattice vertex algebra $M_{H_V}(0)\otimes V_\tw$.
This is possible because we can classify all $H$-lattice vertex subalgebras of a VH pair by Lemma \ref{lattice_correspondence}.
%

Section \ref{sec_coset} is devoted to studying the coset constructions of vertex algebras and VH pairs.
In section \ref{sec_mass}, a mass formula will be proved.


\subsection{Coset constructions} \label{sec_coset}
%
Let $V$ and $W$ be vertex algebras.
For each $V$-module $M$, the tensor product $W \otimes M$ becomes a $W\otimes V$-module. 
We denote by $T_W$ the functor which assigns the $W \otimes V$-module $W \otimes M$ to each $V$-module $M$:
$$
T_{W}:\underline{V \text{-mod}} \rightarrow \underline{W\otimes V\text{-mod}},\ \ M\mapsto W \otimes M.
$$
For a vertex algebra $\Vt$ and a vertex subalgebra $W$ of $\Vt$ and a $\Vt$-module $\Mt$, 
the subspace 
$$C_\Mt(W)=\{v\in \Mt\;|\; w(n)v=0\text{ for any}\; w \in W \text{ and }\; n \geq 0   \}$$
is called a {\it coset} (or commutant).
If $\Mt=\Vt$, viewed as a $\Vt$-module, then $C_{\Vt}(W)$ is a vertex subalgebra of $\Vt$.
It is clear that $C_{\Mt}(W) $ is a $C_\Vt(W)$-module for any $\Vt$-module $\Mt$.
Consider the functor $R_W$ which assigns the $C_{\Vt}(W)$-module $C_{\Mt}(W)$ to each $\Vt$-module $\Mt$,
$$R_W: \underline{\Vt\text{-mod}} \rightarrow \underline{C_{\Vt}(W) \text{-mod}},\; \Mt \mapsto C_\Mt(W).$$
In this subsection, we consider the composition of the coset functor $R_W$ and the tensor product functor $T_W$.
We first consider the case of $R_W \circ T_W$.
If $C_W(W)=\C\1$, then $R_W \circ T_W(M) = C_{W\otimes M}(W)=\C\1\otimes M$ for any $V$-module $M$.
Hence, we have:
\begin{lem}
\label{tensor_natural}
If $C_W(W)=\C\1$, then the composite functor $R_W \circ T_W :\underline{V\text{-mod}} \rightarrow \underline{V\text{-mod}}$
is naturally isomorphic to the identity functor.
\end{lem}

\begin{rem}
We note that $C_W(W)$ is the center of $W$ and the assumption $C_W(W)=\C\1$ is satisfied for almost all vertex algebras.
In particular, if $W$ is a conformal vertex algebra with the conformal vector $\omega$, then $T_W=\omega(0)$
and thus $C_W(W)=\ker T_W$.
\end{rem}

We now consider the case of $T_W \circ R_W$.
Let $\Vt$ be a vertex algebra and $W$ a vertex subalgebra.
Set $W'=C_{\Vt}(W)$.
Since the vertex subalgebras $W$ and $W'$ commute with each other (in $\Vt$),
by the universality of the tensor product in the category of vertex algebras,
there is a vertex algebra homomorphism $i_W: W\otimes W' \rightarrow \Vt,\; (w,w') \mapsto w(-1)w'$, which
implies that $\Mt$ is a $W \otimes W'$-module.
Consider this pullback functor,
$${i_W}^*: \underline{\Vt \text{-module}} \rightarrow \underline{W\otimes W' \text{-module}},\;
 \Mt \mapsto \Mt.$$

We let $f_{\tilde{M}}$ denote the linear map $W\otimes C_{\tilde{M}}(W) \rightarrow \tilde{M}$ defined by setting $f_{\tilde{M}}(w\otimes m) = w(-1)m$ for $w\in W$ and $m \in  C_{\tilde{M}}(W)$.

\begin{lem}
\label{coset_natural}
The family of linear maps $f_{\tilde{M}}$ gives a natural transformation from $T_W \circ R_W$ to the pullback functor 
${i_W}^*$.
\end{lem}

\begin{proof}
It suffices to show that $f_{\tilde{M}}:W\otimes C_{\tilde{M}}(W) \rightarrow \tilde{M}$ is a $W \otimes W'$-module homomorphism (The naturality of $f_{\tilde{M}}$ is clear).
Let $w_1, w_2 \in W$ and $w' \in W'$, $m \in C_{\tilde{M}}(W)$ and $k \in \Z$.
It suffices to show that 
$f_{\tilde{M}}\bigl((w_1\otimes w')(k) (w_2\otimes m)\bigr)= (w_1\otimes w')(k) f_{\tilde{M}}(w_2\otimes m)$.
Since $[w_1(p),w'(j)]=[w_2(p),w'(j)]=0$ and $w_1(q)m=w_2(q)m=0$ for any $j, p\in \Z$ and $q \geq 0$,
we have 
\begin{align}
f_{\tilde{M}}(w_1(i)w_2\otimes w'(j)m)) &=(w_1(i)w_2)(-1)(w'(j)m) \nonumber \\ 
&=\sum_{l \geq 0}\binom{i}{l}(-1)^l \bigl(w_1(i-l)w_2(-1+l) - (-1)^i w_2(i-1-l)w_1(l)\bigr)w'(j)m \nonumber \\
&=w_1(i)w'(j)w_2(-1)m \nonumber
\end{align} for any $i,j \in \Z$.
Hence, 
\begin{align}
f_{\tilde{M}}\bigl((w_1\otimes w')(k) (w_2 \otimes m) \bigr)&=
\sum_{l \in \Z} f_{\tilde{M}}(w_1(l)w_2\otimes w'(k-l-1)m) \nonumber \\
&=\sum_{l \in \Z} w_1(l)w'(k-l-1)w_2(-1)m \nonumber \\
&=(\sum_{l \geq 0} w_1(-1-l)w'(k+l)+w'(k-l-1)w_1(l))w_2(-1)m    \nonumber \\
&=(w_1(-1)w')(k)f_\Mt(w_2 \otimes m)=(w_1\otimes w') (k) f_{\tilde{M}}(w_2\otimes m). \nonumber
\end{align}
\end{proof}

We will say that a vertex algebra $V$ is {\it completely reducible}\/ if every $V$-module is completely reducible. 
If a completely reducible vertex algebra $V$ has a unique simple $V$-module up to isomorphism, we will say that $V$ is {\it holomorphic}. 
Hereafter, we consider the functor $R_W$ and $T_W$ for a holomorphic vertex algebra $W$.
Let $W$ be a vertex subalgebra of $\Vt$ and $\Mt$ a $\Vt$-module. Suppose that $W$ be a holomorphic vertex algebra such that $C_W(W)=\C\1$.
Since $W$ is a holomorphic vertex algebra, as a $W$-module, $\Mt$ is a direct sum of copies of $W$.
Since $C_W(W)=\C\1$,
the natural transformation given in Lemma \ref{coset_natural},
\begin{align}
f_\Mt: W \otimes C_\Mt(W) \rightarrow \Mt, (w,m) \mapsto w(-1)m, \label{eq_split}
\end{align}
is an isomorphism.
In particular, $\Vt$ is isomorphic to $W \otimes C_\Vt(W)$ as a vertex algebra. 
Combining Lemma \ref{tensor_natural} and Lemma \ref{coset_natural} and the above discussion, we have:
\begin{thm}
\label{coset_equivalence}
Let $\Vt$ be a vertex algebra and $W$ be a vertex subalgebra of $V$ such that $C_W(W)=\C\1$.
If $W$ is holomorphic, then $\Vt$ is isomorphic to $W \otimes C_\Vt(W)$ as a vertex algebra
and the functors $T_W$ and $R_W$ are mutually inverse equivalences between
 \underline{\rm$\Vt$-mod} and \underline{\rm$C_\Vt(W)$-mod}.
\end{thm}

\begin{cor}
If $W$ is holomorphic and $C_W(W)=\C\1$, then
a vertex algebra $V$ is simple (resp. completely reducible, holomorphic) if and only if $V\otimes W$ is simple
(resp. completely reducible, holomorphic).
\end{cor}

\begin{rem}
\label{irreducible_tensor}
Let $M$ be an irreducible module of an associative $\C$-algebra $A$ and $N$ an
irreducible module of an associative $\C$-algebra $B$.
Then, it is well-known that $M\otimes N$ is an irreducible module of $A\otimes B$ if both $M$ and $N$ are finite dimensional, which is not true in general. For example,
let $A=B=M=N=\C(x)$. Then, the claim fails, since $A \otimes B$ is not a field.
\end{rem}

For our application, we are interested in a coset of a VH pair by a holomorphic lattice vertex algebra.
Let $(\Vt,H)$ be a good VH pair and $(L_{\Vt,H},H)$ be a maximal lattice.
Let $M$ be a sublattice of $L_{\Vt,H}$ such that $M$ is unimodular.
By Proposition \ref{existence},
the lattice vertex algebra $V_M$, which is holomorphic, is a vertex subalgebra of $\Vt$.
Let $H'$ be the canonical Heisenberg subspace of the lattice vertex algebra $V_M$
and set  $H''=\{h \in H\;|\;(h,h')=0 \text{ for any } h' \in H' \}$.
By Theorem \ref{coset_equivalence}, $\Vt \cong V_M \otimes R_{V_M}(\Vt)$.
Hence, we have:
\begin{lem}
$R_{V_M}(\Vt)=\{v \in \Vt \;|\; h(n)v=0 \text{ for any } n \geq 0 \text{ and } h \in H' \}$.
\end{lem}
Since $M$ is non-degenerate, $H=H'\oplus H''$ and $H'' \subset R_{V_M}(\Vt)$. 
In particular, $(R_{V_M}(\Vt),H'')$ is canonically a VH pair.
Hence, we have:
\begin{lem}
\label{coset_lattice}
If $M$ is a unimodular sublattice of $L_{\Vt,H}$,
then $(\Vt,H_V)$ is isomorphic to $(V_M \otimes R_{V_M}(\Vt), H' \oplus H'')$ as a VH pair.
\end{lem}

Let $\tw$ be the unique even unimodular lattice of signature $(1,1)$;
that is, $\tw$ has a basis $z,w$ such that $(z,z)=(w,w)=0$ and $(z,w)=-1$.
Then, we have:
\begin{lem}
\label{lattice_correspondence}
For a good VH pair $\Vt$, there is a one-to-one correspondence between 
sublattices of $L_{\Vt}$ which are isomorphic to $\tw$ and
decompositions of $\Vt$ into tensor products $\Vt \cong V_\tw \otimes V'$ as VH pairs.
\end{lem}

\subsection{Mass formulae for VH pairs}\label{sec_mass}
By Corollary \ref{tensor_max}, if $V$ and $W$ are in the same genus, then
$L_V\oplus \tw \cong L_W \oplus \tw$, where $L_V$ and $L_W$ are the maximal lattices of $V$ and $W$, respectively.
By Lemma \ref{def_genus}, $L_V$ and $L_W$ are in the same genus of lattices.

Let $f \in \Aut V$ such that $f(H_V)=H_V$. By Lemma \ref{omega_auto},
$f$ induces an automorphism on $\Omega_V$. Furthermore, $f$ induces an automorphism on the maximal lattice $L_V \subset H_V$. The image of such automorphisms forms a subgroup of $\Aut L_V$, denoted by
$G_V$.

For a lattice $L$, set $$\tilde{L}=L\oplus \tw$$ and  $$S_\tw(\Lt)=\{M \subset \Lt\;|\; M \text{ is a rank 2 sublattice which is isomorphic to }\;\tw  \}.$$ 
The group $\Aut \Lt$ naturally acts on $S_\tw(\Lt)$. Let us denote 
the set of orbits by $\Aut \Lt \backslash S_\tw(\Lt)$.
Then, there is a one-to-one correspondence between $\Aut \Lt \backslash S_\tw(\Lt)$ 
and the isomorphism classes in $\gen(L)$.

Since the maximal lattice of $\Vt=V\otimes V_\tw$ is $\widetilde{L_V}=L_V \oplus \tw$,
the group $G_{\Vt}$ acts on $S_{\tw} (\widetilde{L_V})$.
By analogy with the case of lattices, we have:
\begin{prop}
\label{coset_correspondence}
There is a one-to-one correspondence between $G_{\Vt} \backslash S_\tw(\widetilde{L_V})$ 
and the isomorphism classes in $\gen(V)$.
\end{prop}

\begin{proof}
Let $S,S' \in S_\tw(\widetilde{L_V})$.
By Lemma \ref{lattice_correspondence}, there exist the corresponding lattice vertex subalgebras
$V_S \subset \Vt$ and $V_{S'} \subset \Vt$. Set $C_S=C_{\Vt}(V_S)$ and $C_{S'}=C_{\Vt}(V_{S'})$.
Then, $\Vt \cong V_S \otimes C_S$ as a VH pair.
Hence, $C_S \in \gen(V)$.

Suppose that there exists $f \in G_{\Vt}$ such that $f(S)=S'$.
Then, $f(V_S) = V_{S'}$ in $\Vt$. Thus, $f$ induces an isomorphism $C_S \rightarrow C_{S'}$.
Thus, we have a map $$G_{\Vt} \backslash S_\tw(\widetilde{L_V}) \rightarrow \{\text{the VH pair isomorphism classes in }\gen(V)\}, \; S \mapsto C_S.$$
Let $g: C_S \rightarrow C_{S'}$ be an isomorphism of VH pairs.
Take an isomorphism of VH pairs $i_0 :C_S \rightarrow C_{S'}$.
Then, $i_0 \otimes g$ is an isomorphism $V_S \otimes C_S \rightarrow V_{S'}\otimes C_{S'}$ as VH pairs.
Since both $V_S \otimes C_S$ and $V_{S'}\otimes C_{S'}$ are canonically isomorphic to $\tilde{V}$ as VH pairs,
the composite of them is an element of $G_{\Vt}$ that takes $S$ to $S'$.
%
Hence, the map is injective. The surjectivity of it follows from Lemma \ref{lattice_correspondence}.
\end{proof}

In general, $G_{\Vt}$ is a proper subgroup of $\Aut \Lt$.
Hence, it is possible that there are non-isomorphic vertex algebras in a genus whose maximal lattices are the same.
Let $S$ be a sublattice of $\Lt$.
Set $$S^\perp =\{\al \in \Lt \;|\; (\al,\be)=0 \text{ for any } \be \in S \},$$ which is a sublattice of $\Lt$,
and
set $$\mathrm{St}_{\Aut \Lt}(S)=\{f\in \Aut \Lt\;|\; f(S)=S \}.$$

\begin{lem}
\label{orbit_group}
Let $V$ be a good VH pair and $L_V$ be the maximal lattice.
Let $L_0 \in \gen(L_V)$ and choose $S_0 \in S_\tw(\widetilde{L_V})$
 with $S_0^\perp \cong L_0$. Then, there is a one-to-one correspondence between
the isomorphism classes of VH pairs in $\gen(V,H_V)$ whose maximal lattice is isomorphic to $L_0$ 
and the double coset $G_{\Vt} \backslash \Aut (\widetilde{L_V}) / \mathrm{St}_{\Aut \widetilde{L_V}}(S_0)$.
Furthermore, $\Aut \widetilde{L_V}=G_{\Vt}$ if and only if $|G_{\Vt} \backslash \Aut (\widetilde{L_V}) / \mathrm{St}_{\Aut \widetilde{L_V}}(S_0)  |=1$ and $\mathrm{St}_{\Aut \widetilde{L_V}}(S_0) \subset G_{\Vt}$.
In this case, there exists a one-to-one correspondence between $\gen(L_V)$
and $\gen(V,H_V)$. In particular, a VH pair in $\gen(V,H_V)$ is uniquely determined by its maximal lattice.
\end{lem}

\begin{proof}
Set $X = \{S \in S_\tw(\widetilde{L_V}) \;|\; L_0 \cong S^\perp \}$.
Since $\Aut \widetilde{L_V}$ acts on $X$ transitively, $X \cong \Aut \widetilde{L_V} / \mathrm{St}_{\widetilde{L_V}}(S_0)$ as a $\Aut \widetilde{L_V}$-set.
Then, the number of the isomorphism class of VH pairs in $\gen(V)$ whose maximal lattice is isomorphic to $L_0$ is equal to 
$$ |G_{\Vt} \backslash X  |= |G_{\Vt} \backslash \Aut (\widetilde{L_V}) / \mathrm{St}_{\Aut \widetilde{L_V}}(S_0)  |.$$
If $|G_{\Vt} \backslash X  |=1$, then $G_{\Vt}\mathrm{St}_{\Aut \widetilde{L_V}}(S_0) = \Aut (\widetilde{L_V})$.
Thus, the assertion follows.
\end{proof}

We need the following lemmas:
\begin{lem}\label{group_stab_lattice}
Let $L_0 \in \gen(L_V)$ and $S_0 \in S_\tw(\widetilde{L_V})$
 with $S_0^\perp \cong L_0$.
Then, $\mathrm{St}_{\Aut \widetilde{L_V}}(S_0)$ is isomorphic
to $\Aut L_0 \times \Aut \tw$.
\end{lem}
\begin{proof}
Let $f \in \mathrm{St}_{\Aut \widetilde{L_V}}(S_0)$.
Since $f(S_0)=S_0$, we have $f(S_0^\perp)\subset S_0^{\perp}$. Hence, $f$ induces a lattice automorphism of
$S_0^\perp \cong L_0$.
Thus, we have a group homomorphism $\mathrm{St}_{\Aut \widetilde{L_V}}(S_0) \rightarrow \Aut L_0 \times \Aut \tw$,
which is clearly an isomorphism.
\end{proof}
\begin{lem}\label{group_stab}
Let $S \in S_\tw(\widetilde{L_V})$ and
$C_S \in \gen(V)$ be the corresponding VH pair constructed in Proposition \ref{coset_correspondence}.
Then, $\mathrm{St}_{\Aut \widetilde{L_V}}(S) \cap G_{\Vt}$
is isomorphic to $G_{C_S} \times \Aut \tw$.
\end{lem}
\begin{proof}
Let $V_S$ be the lattice vertex algebra constructed in Lemma \ref{lattice_correspondence} associated with $S$
and $f \in \mathrm{St}_{\Aut \widetilde{L_V}}(S) \cap G_{\Vt}$.
Then, by the definition of $G_{\Vt}$,
 there exists $f' \in \Aut \Vt$ such that 
$f'(H_{\Vt})=H_{\Vt}$ and 
the induced automorphism of the maximal lattice $\widetilde{L_V}$ from $f'$ is equal to $f$.
Then, $f'(V_S) = V_S$ and $f'(C_S)=f'(C_{\Vt}(V_S))=C_{\Vt}(V_S)=C_S$. Furthermore, $f'$ induces VH pair automorphisms on $V_S$ and $C_S$.
Thus, $f'$ induces automorphisms on the maximal lattices $S$
and $S^\perp$, and the image in the lattice automorphism groups are in $G_{C_S}$ and $G_{V_S}\cong \Aut \tw$, respectively.
Hence, we have a group homomorphism 
$F_S:\mathrm{St}_{\Aut \widetilde{L_V}}(S) \cap G_{\Vt}
\rightarrow G_{C_S} \times \Aut \tw$,
which is clearly an isomorphism.
\end{proof}

Suppose that $L_V$ is positive-definite and $[\Aut (\widetilde{L_V}):G_{\Vt}]$ is finite.
Then, $G_V$ and $\Aut L_V$ are finite groups and $|G_{\Vt} \backslash \Aut (\widetilde{L_V}) / \mathrm{St}_{\Aut \widetilde{L_V}}(S_0)  |$ is finite for any $L_0 \in \gen(L_V)$. 
We define the mass of $\gen(V,H_V)$ to be
$$\mass(V,H_V) = \sum_{W \in \gen(V)} \frac{1}{|G_{W}|}.$$

By the following proposition, we can calculate the number of
isomorphism classes of VH pairs in a genus with a fixed maximal lattice:
\begin{prop}
\label{mass_fix}
Suppose that $L_V$ is positive-definite and $[\Aut (\widetilde{L_V}):G_{\Vt}]$ is finite. Then, for $L_0 \in \gen(L_V)$, the following equation holds:
$$\frac{[\Aut \widetilde{L_V} : G_{\Vt}]}{|\Aut L_0|} =\sum_{W \in \gen(V),\; L_W \cong L_0 } \frac{1}{|G_W|}.$$
\end{prop}
\begin{proof}
Take $S_0 \in S_\tw(\widetilde{L_V})$ with $S_0^\perp \cong L_0$. Then, by the elementary group theory,
\begin{align*}
[\Aut (\widetilde{L_V}):G_{\Vt}]
&=\sum_{g \in G_{\Vt} \backslash \Aut (\widetilde{L_V}) / \mathrm{St}_{\Aut \widetilde{L_V}}(S_0)}
[\mathrm{St}_{\Aut \widetilde{L_V}}(S_0):\mathrm{St}_{\Aut \widetilde{L_V}}(S_0) \cap g^{-1}G_{\Vt}g].
\end{align*}
Let $g \in \Aut (\widetilde{L_V})$
and $C_{gS_0}$ be the VH pair corresponding to $gS_0 \in  S_\tw(\widetilde{L_V})$.
Then by Lemma \ref{group_stab_lattice} and Lemma \ref{group_stab},
\begin{align*}
[\mathrm{St}_{\Aut \widetilde{L_V}}(S_0):\mathrm{St}_{\Aut \widetilde{L_V}}(S_0) \cap g^{-1}G_{\Vt}g]
&=[g \mathrm{St}_{\Aut \widetilde{L_V}}(S_0)g^{-1}:g\mathrm{St}_{\Aut \widetilde{L_V}}(S_0)g^{-1} \cap G_{\Vt}]\\
&=[\mathrm{St}_{\Aut \widetilde{L_V}}(gS_0):\mathrm{St}_{\Aut \widetilde{L_V}}(gS_0) \cap G_{\Vt}]\\
&=[\Aut (gS_0)^\perp \times \Aut \tw : G_{C_{gS_0}} \times \Aut \tw]\\
&=\frac{ |\Aut L_0| }{|G_{C_{gS_0}}|}.
\end{align*}
Hence, by Lemma \ref{orbit_group}, we have
\begin{align*}
\frac{[\Aut (\widetilde{L_V}):G_{\Vt}]}{|\Aut L_0|}
&=\sum_{g \in G_{\Vt} \backslash \Aut (\widetilde{L_V}) / \mathrm{St}_{\Aut \widetilde{L_V}}(S_0)}
\frac{ 1 }{|G_{C_{gS_0}}|}\\
&=\sum_{W \in \gen(V), L_W\cong L_0}
\frac{ 1 }{|G_{W}|}.
\end{align*}

\end{proof}

Summing up the equation in Proposition \ref{mass_fix} over $L_0 \in \gen(L_V)$, we have:
\begin{thm}\label{mass_vertex}
Let $(V,H_V)$ be a good VH pair.
If $L_V$ is positive-definite and $[\Aut (\widetilde{L_V}):G_{\Vt}]$ is finite,
then $$\mass(V,H_V)=\mass(L_V)[\Aut \widetilde{L_V} : G_{\Vt}].$$
\end{thm}


\section{VH pair with coformal vector}
In this section, VH pairs with a conformal vector are studied.
We first recall the definition of a conformal vertex algebra and a vertex operator 
algebra and its dual module.

A conformal vertex algebra is a vertex algebra $V$ with a vector 
$\omega \in V$ satisfying the following conditions:
\begin{enumerate}
\item
There exists a scalar $c \in \C$ such that 
$$
[L(m),L(n)]=(m-n)L(m+n)+\frac{m^3-m}{12}\delta_{m+n,0}\,c
$$
holds for any $n,m \in \Z$, where $L(n)=\omega(n+1)$;
\item
$L(-1)=T_V$;
\item
$L(0)$ is semisimple on $V$ and all its eigenvalues are integers.
\end{enumerate}
The scalar $c$ of a conformal vertex algebra is called the {\it central charge} of the conformal vertex algebra.
Let $(V,\omega)$ be a conformal vertex algebra and set $$V_n=\{v\in V\;|\; L(0)v=nv \}$$
for $n \in \Z$.
A conformal vertex algebra $(V,\omega)$ is called a {\it vertex operator algebra of CFT type} (hereafter called a {\it VOA}) if it satisfies the following conditions:
\begin{enumerate}
\item[VOA1)]
If $n \leq -1$, then $V_n=0$;
\item[VOA2)]
$V_0 = \C\1$;
\item[VOA3)]
$V_n$ is a finite dimensional vector space for any $n \geq 0$.
\end{enumerate}
In this paper, a conformal vertex algebra (resp. a VOA) of central charge $l$ is called a {\it $c=l$ conformal vertex algebra (resp. a $c=l$ VOA)}.

Let $(V,\omega)$ be a VOA.
Set $V^\vee=\bigoplus_{n \geq 0}V_n^\ast$, where $V_n^\ast$ is the dual vector space of $V_n$.
For $f \in V^\vee$ and $a \in V$, set $$Y_{V^\vee}(a,z)f(-) =  
f(Y(e^{L(1)z}(-z^{-2})^{L(0)}a,z^{-1})-).$$
The following theorem was proved in \cite[Theorem 5.2.1]{FHL}:

\begin{thm}
\label{def_dual}
$(V^\vee,Y_{V^\vee})$ is a $V$-module.
\end{thm}

The $V$-module $V^\vee$ in Theorem \ref{def_dual} is called a
dual module of $V$.
A VOA $(V,\omega)$ is said to be self-dual if $V$ is isomorphic to
$V^\vee$ as a $V$-module.
The following theorem was proved in \cite[Corollary 3.2]{Li}:
\begin{thm}
\label{Li}
Let $(V,\omega)$ be a simple VOA. Then, $(V,\omega)$ is self-dual if and only if $L(1)V_1=0$.
\end{thm}

If $(V,\omega)$ is self-dual, then 
there exists a non-degenerate bilinear form $(,):V \times V \rightarrow \C$
such that
$$(Y(a,z)b,c)=(b,Y(e^{L(1)z}(-z^{-2})^{L(0)}a,z^{-1})c) \label{invariance}$$ for any $a,b,c \in V$.
A bilinear form on $V$ with the above property is called an invariant bilinear form on
$(V,\omega)$.
We remark that $(V,\omega)$ is self-dual if and only if
it has a non-degenerate invariant bilinear form.
A subspace $W \subset V$ is called a subVOA of $(V,\omega)$ if $W$ is a vertex subalgebra and $\omega \in W$.

In this section, we will introduce some technical conditions for a 
conformal vertex algebra with a Heisenberg subspace $H$,
which is an analogue of a self-dual VOA (of CFT type).

Let $(V,\omega)$ be a conformal vertex algebra with a Heisenberg subspace $H \subset V$.
Assume that $L(n)h=0$ and $L(0)h=h$ and
$h(0)$ is semisimple on $V$ for any $h\in H$ and $n\geq 1$.
Then, $h(0)$ commutes with $L(0)$.
For $\al \in H$, set $$V_n^\al=\{v \in V_n\;|\; h(0)v=(h,\al)v \text{ for any }h \in H  \}$$ and $$M_{V,H} =\{\al \in H\;|\; V_n^\al \neq 0 \text{ for some }n \in \Z \}$$
and $$V_H^\vee=\bigoplus_{n \in \Z, \al \in M_{V,H}}(V_n^\al)^*, $$ which is a restricted dual space.
We would like to define a $V$-module structure on $V_H^\vee$ in analogy with Theorem \ref{def_dual}.
The difficulty in defining the dual module structure on a conformal vertex algebra lies in the fact
that it does not always satisfy the following condition:
\begin{itemize}
\item
$L(1)$ is a locally nilpotent operator.
\end{itemize}

We assume that, for any $\al \in M_{V,H}$, $V_n^\al=0$ for sufficiently small $n$.
Then, since $L(1)$ commutes with $h(0)$ for any $h \in H$,
$L(1)$ is locally nilpotent.
By using $V^\al (n) V^\be \subset V^{\al+\be}$,
we can define a $V$-module structure on $V_H^\vee$ similarly to Theorem \ref{def_dual}.
Hence, we have:
\begin{lem}
If a conformal vertex algebra $(V,\omega)$ and a Heisenberg subspace $H\subset V$ satisfy the above two assumptions,
then $V_H^\vee$ is a $V$-module by the formula for $Y_{V^\vee}$ given just before Theorem \ref{def_dual}.
\end{lem}
\begin{rem}
In \cite{HLZ}, Huang, Lepowsky and Zhang consider essentially the same condition for a module of a M{\"o}bius vertex algebra graded by an abelian group.
\end{rem}

\begin{dfn}
\label{Cartan}
Given a conformal vertex algebra $(V,\omega)$ and a Heisenberg subspace $H \subset V$, we say that the triple $(V,\omega,H)$ satisfies assumption (A) if the following conditions are satisfied:
%
\begin{enumerate}
\item[A1)]
$L(n)h=0$ and $L(0)h=h$ for any $n \geq 1$ and $h\in H$;
\item[A2)]
For any $h \in H$, $h(0)$ is semisimple on $V$;
\item[A3)]
$(\al,\be) \in \mathbb{R}$, for any $\al,\be \in M_{V,H}$;
\item[A4)]
If $\frac{(\al,\al)}{2}>n$, then $V_n^\al=0$;
\item[A5)]
$V_0^0=\C\1$ and $V_1^0=H$;
\item[A6)]
$V_n^\al$ is a finite dimensional vector space;
\item[A7)]
$V_H^\vee \cong V$ as a $V$-module.
\end{enumerate}
\end{dfn}
\begin{rem}
Let $(V,\omega,H)$ satisfy Assumption (A).
By (A7), there exists a non-degenerate invariant bilinear form.
Normalize the bilinear form so that $(\1,\1)=-1$.
Let $h,h' \in H$. By the invariance of the form, $(h,h')=(-1)(\1,h(1)h')$. 
Hence, our choice of normalization makes this bilinear form coincide with the form given in the definition of Heisenberg subspace.
\end{rem}

In this section, we study a conformal vertex algebra and a Heisenberg subspace satisfying 
Assumption (A), which is denoted by $(V,\omega,H)$.

Assumption (A) is used in order to show the following Theorems,
which are proved in Section \ref{sec_str}:
\begin{thm}
\label{Cartan_max}
Set $L'_{V,H}=\{\al \in H\;|\; V_{\al^2/2}^\al \neq 0 \}$.
Then, $(\Omega_{V,H})^\lat = \bigoplus_{\al \in L'_{V,H}} V_{\al^2/2}^\al$.
In particular, $L'_{V,H}$ is equal to the maximal lattice $L_{V,H}$ of the VH pair $(V,H)$.
\end{thm}

\begin{thm}
\label{Cartan_Lie}
Let $\al \in M_{V,H}$ such that $\al^2 >0$ and $V_1^\al \neq 0$.
Then, $V_1^\al$ and $V_1^{-\al}$ generate the simple affine vertex algebra $L_{sl_2}(k,0)$
of level $k=2/(\al,\al) \in \Z_{>0}$ and $\dim V_1^{\pm\al}=1$, $\dim V_1^{k\al}=0$ for any
$k \not\in \{ -1,0,1\}$.
Furthermore, $\dim V_n^\be \in \{0,1\}$ for any $\be$ in $M_{V,H}$ with $n \geq \frac{(\be,\be)}{2} > n-1$.
\end{thm}
Those theorems assert that the existence of non-zero vectors in $V_n^\al$ for $\al \in M_{V,H}$
and $n \in \Z$ with $1>n - \al^2/2 \geq 0$ is very important.
It determines the maximal lattice and the existence of affine vertex subalgebras.

Examples and applications of those results are studied in subsection \ref{sec_application};
There, we compute the maximal lattice for a VOA which is an extension of an affine VOA at positive integer level,
together with the mass of the genus for some $c=24$ holomorphic VOAs.

\subsection{Structure of $(V,\omega,H)$}\label{sec_str}
Let $(V,\omega,H)$ satisfy Assumption (A).
In this subsection, we study the structure of $(V,\omega,H)$.
\begin{lem}
\label{Cartan_simple}
If $(V,\omega,H)$ satisfies Assumption (A), then $V$ is a simple vertex algebra and $\ker T= \C\1$. In particular, $(V,H)$ is a good VH pair.
\end{lem}

\begin{proof}
Let $I$ be a non-zero ideal of $V$. Set $I_n^\al=I \cap V_n^\al$. Then, $I=\bigoplus_{n \in \Z, \al \in M_{V,H}}I_n^\al$
since $I$ is a submodule for a diagonalizable action of an abelian Lie algebra.
Hence there exists $\al \in M_{V,H}$ and $n \in \Z$ such that $I_n^\al \neq 0$ and $I_{n-k}^\al=0$ for any $k \geq 1$. Let $v \in I_n^\al$ be a non-zero vector. 
Since $V$ is self-dual, there exists $v' \in V_n^{-\al}$ such that $(v,v')=1$.

Combining $L(1)v=0$ and $v(2n-1)v' \in V_0^0= \C\1$, $(v,v')=1$, we have $\1 \in I$.
Hence, $I=V$.
Since $v(n)a=0$ for any $a \in \ker T$ and $v \in V$ and $n \geq 0$,
$\ker T \subset V_0^0 = \C\1$.
\end{proof}

\begin{cor}
For any  $\al, \be \in M_{V,H}$,
 $\al+\be, -\al \in M_{V,H}$.
\end{cor}

\begin{proof}
Let $\al, \be \in M_{V,H}$.
Since the invariant bilinear form on $V$ induces a non-degenerate pairing, $V_n^\al \times V_n^{-\al} \rightarrow \C$, we have $-\al \in M_{V,H}$.
Let $v \in V^\al$ and $w \in V^\be$ be nonzero elements.
Then, by Lemma \ref{simple_vertex}, we have $0 \neq v(k)w \in V^{\al+\be}$
for some $k \in \Z$. Thus, $\al+\be \in M_{V,H}$. 
\end{proof}

By Remark \ref{omega_simple} and Lemma \ref{Cartan_simple}, 
$\Omega_{V,H}$ is a good AH pair and $\Omega_{V,H}^\lat$ is a lattice
pair.
Set $L'_{V,H}=\{\al \in H\;|\; V_{\al^2/2}^\al \neq 0 \}$.
Before proving Theorem \ref{Cartan_max},
we show the following lemma:
\begin{lem}
For any $\al \in L'_{V,H}$, $\Omega_{V,H}^{\al} = V_{\al^2/2}^\al$.
\end{lem}

\begin{proof}
Let $\al \in L'_{V,H}$ and $v \in V_{\al^2/2}^\al$ be a nonzero element.
Then $h(0)v=(\al,h)v$.
By (A4) $h(n)v=0$ and and $L(n)v=0$ for any $n \geq 1$ and $h \in H$. It suffices to show that $\omega_H(0)v=v(-2)\1$.
Since the invariant bilinear form on $V$ induces a non-degenerate pairing on $V_n^\al \times V_n^{-\al} \rightarrow \C$, there exists $v' \in V_{\al^2/2}^{-\al}$ such
that $(v,v')=(-1)^{\al^2/2+1}$.
By the invariance of the form, we have 
$-1=(v,v')(-1)^{\al^2/2}=(\1,v((\al,\al)-1)v')$.
Since $v((\al,\al)-1)v' \in V_0^0=\C\1$, we have $v((\al,\al)-1)v'=\1$.
Since $v((\al,\al)-2)v' \in V_1^0= H$ and
$(h,v((\al,\al)-2)v')=(-1)(\1,h(1)v((\al,\al)-2)v'
=(-1)(h,\al)(\1, v((\al,\al)-1)v')=(h,\al)$, we have $v((\al,\al)-2)v'=\al$.

Since $v'((\al,\al)+k)v \in V_{-k-1}^0=0$ and $v(-(\al,\al)+k)v \in V_{2(\al,\al)-k-1}^{2\al}=0$ for any $k \geq 0$,
by applying the Borcherds identity with $(p,q,r)=(-(\al,\al), (\al,\al)-1,(\al,\al)-2)$,
we have
\begin{align*}
\al(-1)v &=
\sum_{i \geq 0} \binom{-(\al,\al)}{i}(v((\al,\al)-2+i)v')(-1-i)v \nonumber \\
&=\sum_{i\geq 0} (-1)^i\binom{(\al,\al)-2}{i}(v(-2-i)v'((\al,\al)-1+i)v
-v'(2(\al,\al)-3-i)v(-(\al,\al)+i))v \nonumber \\
&= v(-2)\1. \nonumber
\end{align*}
Hence, we have $V_{\al^2/2}^\al \subset \Omega_{V,H}^\al$.
By Lemma \ref{lattice_pair_twisted}, $\dim \Omega_{V,H}^\al=1$.
Thus, we have $\Omega_{V,H}^{\al} = V_{\al^2/2}^\al$.
\end{proof}

\begin{proof}[proof of Theorem \ref{Cartan_max}]
Let $\al \in L_{V,H}$, that is, $\Omega_{V,H}^\al, \Omega_{V,H}^{-\al} \neq 0$.
By the above lemma, $L'_{V,H} \subset L_{V,H}$.
Thus, it suffices to show that $\al \in L'_{V,H}$.
Let $v \in \Omega_{V,H}^\al$ and $v' \in \Omega_{V,H}^{-\al}$ be nonzero elements. 
Then, $v=\sum_i v_i$ where $v_i \in V_i^\al$.
Since $h(n)V_i \subset V_{i-n}$ and $TV_i \subset V_{i+1}$,
we have $v_i \in \Omega_{V,H}^\al$.
Since $\dim \Omega_{V,H}^\al=1$, we have $\Omega_{V,H}^\al \subset V_k^\al$ and $\Omega_{V,H}^{-\al} \subset V_{k'}^{-\al}$ for some $k,k' \in \Z$.
By Lemma \ref{product}, we can assume that $\1=v((\al,\al)-1)v' \in \Omega_{V,H}^0=\C\1$. Since $v((\al,\al)-1)v' \in V_{k+k'-(\al,\al)}$ and $k,k' \geq \al^2/2$,
 we have $k=k'=\al^2/2$. Thus, $\al \in L'_{V,H}$.
\end{proof}
We now show that $(\al,\be) \in \Z$ and $(\al,\al)\in 2\Z$ for any $\al ,\be \in L'_{V,H}$.
Recall that $M_{V,H}=\{\al \in H\;|\; V_n^\al \neq 0 \text{ for some }n \in \Z \}$.
Let $\al \in L_{V,H}$ such that $\al^2 \neq 0$.
Then, $V_{\al^2/2}^{\pm \al}$ generates the lattice vertex algebra $V_{\Z\al}$
and $V$ is a direct sum of irreducible $V_{\Z\al}$-modules.
Hence, we have:
\begin{cor}
Let $\al \in L_{V,H}$ such that $\al^2 \neq 0$.
Then, $(\al,\be) \in \Z$ for any $\be \in M_{V,H}$.
\end{cor}
Set $L_{V,H}^\vee = \{\al \in H\;|\; (\al,\be) \in \Z \text{ for any }\be
\in L_{V,H} \}$.

\begin{cor}
If $L_{V,H}$ is non-degenerate, then $M_{V,H} \subset L_{V,H}^\vee$.
\end{cor}

We will use the following lemma:
\begin{lem}
\label{positive_VOA}
Suppose that $L_{V,H}$ is positive-definite and spans $H$.
Then, $V$ is a VOA (of CFT type) and $M_{V,H}$ is a subgroup of $L_{V,H}^\vee$.
\end{lem}
\begin{proof}
Since $L_{V,H}$ is non-degenerate and spans $H$,
$M_{V,H} \subset L_{V,H}^\vee \subset \mathrm{span}_{\mathbb{R}} L_{V,H}$, which implies that 
$M_{V,H}$ is positive-definite.
If $0 \neq \al \in M_{V,H}$, $V_n^\al =0$ for any $n \leq 0$.
Hence, $V=\bigoplus_{n \geq 0} V_n$ and $V_0=\C\1$.
Since $\{ \al \in M_{V,H} \,|\; \al^2 < n  \}$ is a finite set for any $n > 0$,
$\dim V_n$ is finite.
\end{proof}

Now, we will show Theorem \ref{Cartan_Lie}:
\begin{proof}[proof of Theorem \ref{Cartan_Lie}]
Let $0 \neq e \in V_1^\al$ and $f \in V_1^{-\al}$ such that $(e,f)=1$.
Since $\al^2 >0$, $L(n)e \in V_{1-n}^\al=0$ for any $n \geq 1$.
Since $e(0)f \in V_1^0=H$ and $(h,e(0)f)=(h,\al)(e,f)=(h,\al)$, we have $e(0)f=\al$.
Since $V_{-n}^{2\al}=0$ for any $n \geq 0$, by skew-symmetry of vertex algebras, we have $e(0)e=-e(0)e+Te(1)e+\cdots=-e(0)e$.
Thus, $\{ e, \frac{2}{\al^2}f,\frac{2}{\al^2}\al \}$ is a $sl_2$-triple under the $0$-th product.
Similarly, we have $e(n)e=f(n)f=0, h(n)e=h(n)f=0$ for any $n \geq 1$.
Since $1=(e,f)=-(\1,e(1)f)$, we have $e(1)f=\1$. 
Thus, $\{ e, \frac{2}{\al^2}f,\al\}$  generates an affine vertex algebra $W$ of level $k=2/(\al,\al)$.
Since $e(-1)^n\1 \in V_{n}^{n\al}=0$ for any $n^2\al^2 >n$,
$W$ is isomorphic to the simple affine vertex algebra $L_{sl_2}(k,0)$ and the level $k$ is a positive integer \cite[Theorem 3.1.2]{FZ}.
We will show that $\dim V_1^{-\al} =1$.
Let $f' \in V_1^{-\al}$ such that $(e,f')=0$. Then, $e(0)f'=0$ follows from the computation presented above.
Since $V_1^{\pm n\al}=0$ for sufficiently large $n$, 
$f(0)^n f'=0$. Hence, $f'=0$ follows from the representation theory of the Lie algebra $sl_2$.
Hence, $\dim V_1^{-\al}=1$.
Since $\dim V_1^{\pm \al}=1$, for any $k \in \Z$ with $k\neq \pm 1 ,0$, $\dim V_1^{k\al}=0$ follows from the representation theory of the Lie algebra $sl_2$ again.
Finally, suppose $V_n^\be \neq 0$ for $\be$ in $M_{V,H}$ with $n \geq \frac{(\be,\be)}{2} > n-1$.
Then, by Lemma \ref{tensor_A}, $V \otimes V_\tw$ satisfies Assumption (A).
Let $\gamma \in \tw$ such that $\frac{(\gamma,\gamma)}{2}=-n+1$.
Since $0 \neq V_n^\be \otimes {V_\tw}_{-n+1}^\gamma   \subset (V\otimes V_\tw)_1^{(\be,\gamma)}$
and $1\geq \frac{((\be,\gamma),(\be,\gamma))}{2}=\frac{(\be,\be)}{2}+\frac{(\gamma,\gamma)}{2} >0$,
by the above result, we have $\dim  (V\otimes V_\tw)_1^{(\be,\gamma)} =1$.
Thus, $\dim V_n^\be=1$.
\end{proof}

For $\al \in M_{V,H}$ with $\al^2 >0$ and $V_1^\al \neq 0$,
define $r_\al:H\rightarrow H$ by $$r_\al(h)=h-2\frac{(h,\al)}{\al^2}\al,$$
and
set $$\tilde{r}_\al=\exp(\frac{2}{\al^2}f(0))\exp(-e(0))\exp(\frac{2}{\al^2}f(0)),
$$
where $e$ and $f$ is the same in the proof of Theorem \ref{Cartan_Lie}.
The following Corollary follows from the representation theory of affine $sl_2$
and Theorem \ref{Cartan_Lie}:
\begin{cor}
\label{affine_reflection}
For $\al \in M_{V,H}$ with $\al^2>0$ and $V_1^\al \neq 0$,
 $\tilde{r}_\al$ is a vertex algebra automorphism of $V$ satisfying the following conditions:
\begin{enumerate}
\item
$\tilde{r}_\al(\omega)=\omega$;
\item
$\tilde{r}_\al(H)=H$;
\item
$\tilde{r}_\al|_H=r_\al$.
\end{enumerate}
In particular, the reflection $r_\al$ is an automorphism of the maximal lattice 
$L_{V,H}$. 
\end{cor}

\begin{lem}
If $\al^2 >0$ and $V_1^\al \neq 0$, then $\frac{2}{\al^2}\al \in L_V$.
\end{lem}

\begin{proof}
Let $e \in V_1^\al$ be a nonzero element
 and set $k=2/\al^2$.
By the representation theory of the Lie algebra $sl_2$,
$0 \neq e(-1)^k\1 \in V_k^{k\al}$.
Since $(k\al)^2/2=k$, we have $k\al \in L_V$.
\end{proof}

We end this subsection \ref{sec_str} by mentioning a generalization
of the theory of genera of VH pairs to the triples $(V,\omega,H)$.
The results in Section \ref{sec_genus} are also valid for the triple $(V,\omega,H)$ with minor changes.

\begin{lem}
\label{tensor_A}
Let $(V,\omega_V, H_V)$ and $(W,\omega_W,H_W)$ satisfy Assumption (A).
Then, $(V\otimes W , \omega_V+\omega_W,H_V\oplus H_W)$ satisfy Assumption (A).
\end{lem}

\begin{proof}
Let $\al \in H_V$ and $\al' \in H_W$.
Then, $(V\otimes W)_n^{(\al,\al')}=\bigoplus_{n=k+k'} V_k^\al \otimes W_{k'}^{\al'}$.
Hence, Definition $(A1),\ldots, (A6)$ is easy to verify.
Since $(V\otimes W)_n^{(\al,\al')}$ is finite dimensional,
$(V\otimes W)_{H_V\oplus H_W}^* \cong V_{H_V}^* \otimes {W}_{H_W}^* \cong V \otimes W$.
\end{proof}

The following lemma is clear from the definition:
\begin{prop}
\label{lattice_vertex_Cartan}
Let $L$ be a non-degenerate even lattice and $V_L$ the lattice conformal vertex algebra. Set $H_L=L \otimes_\Z \C \subset (V_L)_1$. Then, $(V_L,\omega_{H_L},{H_L})$ satisfy
Assumption (A).
\end{prop}

Results in Section \ref{sec_genus} are obtained as follows.
Hereafter, we assume that any triple $(V,\omega,H)$ satisfy Assumption (A).
By Lemma \ref{tensor_A}, the category of triples $(V,\omega,H)$ forms a symmetric monoidal category,
where a morphism from $(V,\omega_V,H_V)$ to $(W,\omega_W,H_W)$ is a vertex algebra homomorphism $f:V\rightarrow W$
such that $f(\omega_V)=\omega_W$ and $f(H_V)=H_W$.
We also remark that Lemma \ref{lattice_correspondence} can be generalized to triples $(V,\omega,H)$.
Hence, we can define a genus of $(V,\omega,H)$ similarly to Section \ref{sec_genus}.
More precisely, 
triples $(V,\omega_V,H_V)$ and $(W,\omega_W,H_W)$ are in the same
genus if there exists a vertex algebra isomorphism $f:V \otimes V_\tw \rightarrow 
W \otimes V_\tw$ such that $f(\omega_{V\otimes V_\tw})=\omega_{W\otimes V_\tw}$ and $f(H_{V \otimes V_\tw})=H_{W \otimes V_\tw}$.
Denote by $\gen(V,\omega_V,H_V)$ the genus of $(V,\omega_V,H_V)$.
Denote by $G_{(V,\omega_V,H_V)}$ the image of automorphisms $f \in \Aut (V)$ such that $f(\omega_V)=\omega_V$ and $f(H_V)=H_V$ in $\Aut L_{V,H_V}$.
If $L_V$ is positive-definite and $[\Aut \widetilde{L_V} : G_{(\Vt,\tilde{\omega}),H_{\Vt}}]$ is finite, then the mass of $\gen(V,\omega_V,H_V)$ is defined to be
$$\mass(V,\omega_V,H_V) = \sum_{(V',{\omega'}_V',H_{V'}) \in \gen(V,\omega_V,H_V)} \frac{1}{|G_{(V',{\omega'}_{V'},H_{V'})}|}.$$

Then, we have:
\begin{thm}
\label{mass}
$\mass(L_V)[\Aut \widetilde{L_V} : G_{(\Vt,\tilde{\omega}),H_{\Vt}}]=\mass(V,
\omega_V, H_V)$.
\end{thm}

We remark on the genus of VOAs.
Let $V$ be a VOA. Assume that two Heisenberg subspaces $H$ and $H'$ satisfy
Assumption (A).
Then, $H$ and $H'$ are split Cartan subalgebras of Lie algebra $V_1$.
Hence, there exists $a_1, \ldots, a_k \in V_1$ such that $f(H)=H'$, where $f=\exp(a_1(0)) \cdots \exp(a_n(0))$ (see, for example, \cite{Hu}).
Since $\exp (a_i(0))$ is a VOA automorphism of $V$, we have:
\begin{lem}
\label{unique_Cartan}
All Heisenberg subspaces of a VOA which satisfies Assumption (A) are conjugate under the VOA automorphism group.
\end{lem}
The lemma asserts that the genus of a VOA is independent of the choice of a Heisenberg subspace $H$.

Similarly to the proof of Lemma \ref{positive_VOA} and by Theorem \ref{coset_equivalence}, we have:
\begin{lem}
\label{VOA_positive_genus}
Let $(V,\omega,H)$ satisfy Assumption (A).
Suppose that $L_{V,H}$ is positive-definite and spans $H$. Then, all conformal vertex algebras in $\gen(V,\omega, H_V)$ are VOAs.
\end{lem}

\subsection{Application to extensions of affine VOAs}\label{sec_application}
In this subsection, we prove that many important conformal vertex algebras satisfy
Assumption (A) and study the maximal lattices for those vertex algebras.

The following Proposition follows from Theorem \ref{Li}:
\begin{prop}
\label{zero_Cartan}
Let $(V,\omega)$ be a conformal vertex algebra.
Then, $(V,\omega,0)$ satisfies Assumption (A) if and only if 
$V$ is a simple self-dual VOA with $V_1=0$.
\end{prop}
In order to prove extensions of a simple affine VOAs at positive integer level satisfy
Assumption (A), we recall some results on simple affine VOA at positive integer level.
%
%
Let $\mathfrak{g}$ be a simple Lie algebra and $\mathfrak{h}$ a Cartan subalgebra
and $\Delta$ the root system of $\mathfrak{g}$.
Let $Q_\mathfrak{g}$ be the sublattice of the weight lattice spanned by long roots and $\{\al_1,\ldots,\al_l\}$ a set of simple roots. Let $\theta$ be the highest root
and $a_1,\ldots,a_l$ integers satisfying $\theta=\sum_{i=1}^l a_i \al_i$.
Let $\langle,\rangle$ be an invariant bilinear form on $\mathfrak{g}$, which
we normalize by $\langle\theta, \theta \rangle=2$.

The following lemma easily follows:
\begin{lem}
\label{long_lattice}
The sublattice $Q_\mathfrak{g} \subset \mathfrak{h}$ is generated
by $\{\frac{2}{\al^2}\al \;|\; \al \in \Delta \}$.
If $\mathfrak{g}$ is $A_n,D_n,E_n$ (resp. $B_n, C_n,F_4,G_2$),
then $Q_\mathfrak{g}$ is isomorphic to the root lattice
(resp.  $D_n$, $A_1^n$, $D_4$, $A_2$).
\end{lem}

Let $L_\mathfrak{g}(k,0)$ denote the simple affine VOA of level $k\in \C$
associated with $\mathfrak{g}$ \cite{FZ}.
Then, the VOA $L_\mathfrak{g}(k,0)$ is completely reducible if and only if
$k$ is a non-negative integer \cite{FZ,DLM}.
We assume that $k$ is a positive integer.

Suppose that a dominant weight $\lambda$ of $\mathfrak{g}$ satisfies $\langle \theta,\lambda \rangle \leq k$.
Then, the highest weight module $L_\mathfrak{g}(k,\lambda)$ is an
 irreducible module of $L_\mathfrak{g}(k,0)$.
Conversely, any irreducible module of $L_\mathfrak{g}(k,0)$ is isomorphic to the module of this form. 
Let $P_{\mathfrak{g},k}^+$ be the set of all dominant weights $\lambda$ satisfying $\langle \theta,\lambda \rangle \leq k$.
For $\al \in \mathfrak{h}$ and $n \in \Z$,
set $L_\mathfrak{g}(k,\lambda)_n^\al=
\{v \in L_\mathfrak{g}(k,\lambda)\;|\; h(0)v=\langle h,\al \rangle v \text{ and } L(0)v=nv   \text{ for any }h\in \mathfrak{h}\}$.
Then, $L_\mathfrak{g}(k,\lambda)=\bigoplus_{\al \in \mathfrak{h}, n\in \Z} L_\mathfrak{g}(k,\lambda)_n^\al$.
%
%
Set $I_\mathfrak{g}=\{0\} \cup \{\Lambda_i \;|\;a_i=1 \text{ and } i \in\{1, \ldots, n \} \}$,
where $\{ \Lambda_i\}_{i \in \{1,\dots, n\}}$ is the fundamental weights of $\{ \mathfrak{g},\al_1,\dots,\al_l\}$.
The following lemma follows from \cite{DR}:
\begin{lem}
\label{affine_weight}
Let $\lambda \in P_{\mathfrak{g},k}^+$.
If $L_\mathfrak{g}(k,\lambda)_n^\al \neq 0$, then
$n \geq \langle \al, \al \rangle/2k$ and the equality holds if and only if 
 $\lambda \in kI_\mathfrak{g}$ and
$\al \in \lambda + kQ_\mathfrak{g}$.
\end{lem}
\begin{rem}
\label{simple_current}
Non-zero weights in $I_\mathfrak{g}$ are called cominimal and $L_\mathfrak{g}(k,k\lambda)$
is known to be a simple current for $\lambda \in I_\mathfrak{g}$ (see \cite[Proposition 3.5]{Li2}).
In fact, $I_\mathfrak{g}$ form a group by the fusion product.
Let $R_\mathfrak{g}$ be a root lattice of $\mathfrak{g}$ and $R_\mathfrak{g}^\vee$ the dual lattice of the root lattice.
Then, there is a natural group isomorphism between $I_\mathfrak{g}$ and $R_\mathfrak{g}^\vee/R_\mathfrak{g}$.
\end{rem}

By Lemma \ref{affine_weight}, we have:
\begin{prop}
\label{affine_Cartan}
Let $(V,\omega)$ be a VOA.
If $L_\mathfrak{g}(k,0)$ is a subVOA of $V$ and $V_1=L_\mathfrak{g}(k,0)_1=\mathfrak{g}$, then
$(V,\omega,H)$ satisfy Assumption (A),
where $H$ is a Cartan subalgebra of $\mathfrak{g}=V_1$.
\end{prop}

Combining Lemma \ref{affine_weight} with Theorem \ref{Cartan_max},
we have:
\begin{cor}
For $k \in \Z_{>0}$ and simple Lie algebra $\mathfrak{g}$,
the maximal lattice of the simple affine VOA 
$L_\mathfrak{g}(k,0)$ is $\sqrt{k}Q_\mathfrak{g}$.
\end{cor}

\begin{rem}
This corollary implies that a level one simple affine vertex algebra associated with a simply laced simple Lie algebra is isomorphic to the lattice vertex algebra.
\end{rem}

Herein, we concentrate on the case that
a VOA $(V,\omega)$ is an extension of an affine VOA at positive integer level.
Let $(V,\omega)$ be a VOA and
$\mathfrak{g}_i$ be a finite dimensional semisimple Lie algebra and
$k_i \in \Z_{>0}$ for $i=1,\dots,N$.
Set $\mathfrak{g}=\bigoplus_{i=1}^N \mathfrak{g_i}$ and let 
$H$ be a Cartan subalgebra of $\mathfrak{g}$.
Assume that $\bigotimes_{i=1}^N L_{\mathfrak{g_i}}(k_i,0)$ is a subVOA of $(V,\omega)$ such that $V_1=\mathfrak{g}$.
Then, it is clear that $(V,\omega,H)$ satisfy Assumption (A). 
We will describe the maximal lattice of $(V,H)$ using Lemma \ref{affine_weight}.

We simply write $L_{\mathfrak{g}}(\vec{k},\vec{\lambda})$ for 
$\bigotimes_{i=1}^N L_{\mathfrak{g_i}}(k_i,\lambda_i)$, where $\vec{\lambda}=(\lambda_1,\dots,\lambda_N) \in P=\bigoplus_{i=1}^N P_{\mathfrak{g_i},k}^+$.
Since $L_{\mathfrak{g}}(\vec{k},0)$ is completely reducible,
$V =\bigoplus_{\vec{\lambda} \in P} L_{\mathfrak{g}}(\vec{k},\vec{\lambda})^{n_{\vec{\lambda}}}$,
where $n_{\vec{\lambda}} \in \Z_{\geq 0}$ is the multiplicity of the module $L_{\mathfrak{g}}(\vec{k},\vec{\lambda})$.
Let $Q_V$ be the sublattice of $H$ generated by $\bigoplus_{i=1}^N \sqrt{k_i}Q_\mathfrak{g_i}$
and all $\frac{1}{\sqrt{k}}\vec{\lambda}$ satisfying the following conditions:
\begin{enumerate}
\item
$(\lambda_1,\dots,\lambda_N) \in P$;
\item
$n_{\vec{\lambda}} >0$;
\item
$\lambda_i \in k_i I_{\mathfrak{g_i}}$ for any $i=1,\dots,N$.
\end{enumerate}
Then, we have:
\begin{prop}
\label{affine_max}
The maximal lattice $L_{V,H}$ is equal to $Q_V$.
\end{prop}

We will consider two examples whose genera are non-trivial.
Set $A_{\mathfrak{g},\vec{k}}= \bigoplus_{i=1}^N I_\mathfrak{g_i}$, which is a group by Remark \ref{simple_current}.
Consider the quadratic form on $A_{\mathfrak{g},\vec{k}}$ defined by
$$A_{\mathfrak{g},\vec{k}} \rightarrow \mathbb{Q}/\Z,\; 
\vec{\lambda} \mapsto \frac{1}{2}\sum_{i=1}^N k_i(\lambda_i,\lambda_i),$$
for $\vec{\lambda} \in A_{\mathfrak{g},k}$.
A subgroup $N \subset A_{\mathfrak{g},\vec{k}}$ is called isotropic
if the square norm of any vector in $N$ is zero.
Let $N$ be an isotropic subgroup of $A_{\mathfrak{g},\vec{k}}$.

Set \begin{align}
L_\mathfrak{g}(\vec{k},N)=\bigoplus_{\vec{\lambda} \in N} L_\mathfrak{g}(\vec{k},k\vec{\lambda}).
\nonumber
\end{align}
Then, by \cite[Theorem 3.2.14]{Ca}, $L_\mathfrak{g}(\vec{k},N)$ inherits a unique simple vertex operator algebra structure.
The assumption of the evenness in the theorem follows from the fact that $L_\mathfrak{g}(\vec{k},0)$ is unitary \cite{CKL}.
\begin{prop}
The maximal lattice of $L_\mathfrak{g}(\vec{k},N)$ is generated by $\bigoplus_{i=1}^N \sqrt{k_i}Q_\mathfrak{g_i}$ and $N$.
\end{prop}

For example, the maximal lattice of $L_{B_{12}}(2,0)$ is $\sqrt{2}D_{12}$, where $D_{12}$ is a root lattice of type $D_{12}$.
The $\gen(\sqrt{2}D_{12})=\genus_{12,0}(2_\genus^{-10}4_\genus^{-2})$ (see Section \ref{subsec_lattice}) contains two isomorphism classes
of lattices, $\sqrt{2}D_{12}$ and $\sqrt{2}E_8D_4$.
Hence, the genus of a VOA $\gen(L_{B_{12}}(2,0))$ contains at least two non-isomorphic VOAs.

Hereafter, we give another class of examples for which we can completely determine the genera of VOAs.
We first recall $c=24$ holomorphic VOAs.
\begin{rem}
Our definition of ``holomorphic'' is slightly different from the usual one.
The usual definition of a holomorphic VOA requires completely reducibility for admissible modules, rather than all weak modules.  
Our definition together with the assumption that a VOA is of CFT type corresponds to strongly holomorphic in the literature.
\end{rem}
For a VOA $V$, the $0$-th product gives a Lie algebra structure on $V_1$.
The rank of a finite dimensional Lie algebra is the dimension of a Cartan subalgebra
of the Lie algebra.
The following theorem was proved in \cite{DM1,Sch,EMS,DM2}:
\begin{thm}
\label{holomorphic}
Let $V$ be a holomorphic $c=24$ VOA.
Then, the Lie algebra $V_1$ is reductive, and exactly one of the following holds:
\begin{enumerate}
\item
$V_1 = 0$;
\item 
$V_1$ is abelian of rank $24$ and $V$ is isomorphic to the lattice VOA of the Leech lattice;
\item
 $V_1$ is one of $69$ semisimple Lie algebras listed in the table below.
\end{enumerate}
Furthermore, if $V_1 \neq 0$, then the vertex subalgebra generated by $V_1$ is a subVOA of $V$.
\end{thm}

\begin{minipage}{0.3\hsize}
\begin{center}
\begin{tabular}{|c|c|}
\hline
rank & Lie algebra\\
\hline
\hline

$24$ &  $U(1)^{24},(D_{4,1})^6, (A_{7,1})^2(D_{5,1})^2, 
(A_{1,1})^{24},(A_{3,1})^8, (A_{4,1})^6, (A_{6,1})^4, (A_{8,1})^3, (D_{6,1})^4$\\ \hline

&$(E_{6,1})^4,A_{11,1}D_{7,1}E_{6,1}, (A_{12,1})^2, A_{15,1}D_{9,1},D_{10,1}(E_{7,1})^2,A_{17,1}E_{7,1},
(D_{12,1})^2,A_{24,1}$\\ \hline
&$(E_{8,1})^3,D_{16,1}E_{8,1}, (A_{5,1})^4D_{4,1},(A_{9,1})^2D_{6,1},
(D_{8,1})^3,D_{24,1},A_{2,1}^{12}
$  \\ \hline \hline

$16$ & $ (A_{3,2})^4(A_{1,1})^4,(D_{4,2})^2(C_{2,1})^4,(A_{5,2})^2C_{2,1}(A_{2,1})^2, (D_{5,2})^2(A_{3,1})^2, A_{7,2}(C_{3,1})^2A_{3,1}$\\ \hline
&$(C_{4,1})^4,D_{6,2}C_{4,1}(B_{3,1})^2,A_{9,2}A_{4,1}B_{3,1},
E_{6,2}C_{5,1}A_{5,1},C_{10,1}B_{6,1},E_{8,2}B_{8,1}$ \\ \hline
& $D_{8,2}(B_{4,1})^2,(C_{6,1})^2B_{4,1}, D_{9,2}A_{7,1}, C_{8,1}(F_{4,1})^2, E_{7,2}B_{5,1}F_{4,1}, A_{1,2}^{16}$\\ \hline \hline
$12$ &  $(A_{1,4})^{12},B_{12,2}, (B_{6,2})^2,B_{4,2}^3, (B_{3,2})^4,(B_{2,2})^6,A_{8,2}F_{4,2}, C_{4,2}(A_{4,2})^2, D_{4,4}(A_{2,2})^4$
\\ \hline

$12$ & $(A_{2,3})^6, E_{7,3}A_{5,1}, A_{5,3}D_{4,3}(A_{1,1})^3,A_{8,3}(A_{2,1})^2, E_{6,3}(G_{2,1})^3,D_{7,3}A_{3,1}G_{2,1}
    $\\ \hline \hline
$10$ & $A_{7,4}(A_{1,1})^3, D_{5,4}C_{3,2}(A_{1,1})^2,
E_{6,4}C_{2,1}A_{2,1}, C_{7,2}A_{3,1},A_{3,4}^3A_{1,2}
$ \\ \hline \hline
$8$ & $A_{4,5}^2, D_{6,5}(A_{1,1})^2$ \\ \hline 
$8$ & $A_{5,6}C_{2,3}A_{1,2}, C_{5,3}G_{2,2}A_{1,1}$ \\ \hline \hline
$6$ & $ A_{6,7}$ \\ \hline 
$6$ & $ F_{4,6}A_{2,2}, D_{4,12}A_{2,6}$ \\ \hline
$6$ & $ D_{5,8}A_{1,2}$ \\ \hline \hline
$4$ & $ C_{4,10}$ \\ \hline \hline
\end{tabular}
\end{center}
\end{minipage}
\\

For a Lie algebra $\mathfrak{g}$ listed above,
let $V_{\mathfrak{g}}^{hol}$ denote a holomorphic VOA whose Lie algebra is $\mathfrak{g}$.
Recently, the existence and uniqueness of those holomorphic VOAs are proved except for the case that $V_1=0$.
In this paper, we will use the following result:
\begin{thm}[\cite{LS,LS2}]
\label{unique_hol}
There exists a unique holomorphic VOA of central charge $24$ whose Lie algebra is $E_{8,2}B_{8,1}$.
\end{thm}

By Theorem \ref{holomorphic} and Proposition \ref{zero_Cartan} and Proposition \ref{affine_Cartan},
we have:
\begin{cor}
\label{hol_Cartan}
Any $c=24$ holomorphic VOA satisfy Assumption (A).
\end{cor}

The following lemma follows from Theorem \ref{coset_equivalence}:
\begin{lem}
Let $(V,\omega,H)$ be a $c=l$ holomorphic conformal vertex algebra satisfying Assumption (A).
Then, any $W \in \gen(V,\omega,H)$ is a $c=l$ holomorphic conformal vertex algebra satisfying Assumption (A).
\end{lem}
Since the maximal lattice of any $c=24$ holomorphic VOA is positive-definite,
we have:
\begin{prop}
\label{coset_VOA}
If $(V,\omega)$ is a $c=24$ holomorphic VOA,
then all conformal vertex algebras in $\gen(V,\omega)$ are $c=24$ holomorphic VOAs.
\end{prop}

The above proposition gives us a effective method to construct holomorphic VOAs.
In the rest of this paper, we calculate the mass of the VOA $V_{E_{8,2}B_{8,1}}^{hol}$ in Theorem \ref{unique_hol}.
By Proposition \ref{affine_max}, the maximal lattice of $V_{E_{8,2}B_{8,1}}^{hol}$ is
$\sqrt{2}E_8 D_8$ and its genus is $\genus_{16,0}(2_\genus^{+10})$.
Since $\genus_{16,0}(2_\genus^{+10})$ contains $17$ isomorphism classes of lattices,
by lemma \ref{coset_VOA} and Theorem \ref{mass},
there exist at least $17$ $c=24$ holomorphic VOAs whose maximal lattices are
in $\genus_{16,0}(2_{\genus}^{+10})$.

The following lemma follows from Lemma \ref{Cartan_Lie}:
\begin{lem}
\label{lem_level1}
Let $(V,\omega,H)$ satisfy Assumption (A).
If $V$ is a VOA and $V_1$ is a semisimple Lie algebra, then there is a one-to-one correspondence between norm $2$ vectors in $L_{V,H}$ and long roots of a level $1$ component of $V_1$.
\end{lem}

Hence, the level $1$ components of $V_1$ are determined by the maximal lattice $L_{V,H}$.
Set $$V_{\genus_{17,1}(2_{\genus}^{+10})}^{hol} = V_{E_{8,2}B_{8,1}}^{hol} \otimes V_\tw,$$
and set $$\genus_{17,1}(2_{\genus}^{+10})=\tw\oplus \sqrt{2}E_8D_8.$$
We prove the following characterization result for $V_{\genus_{17,1}(2_{\genus}^{+10})}^{hol}$.
\begin{thm}
\label{char_18}
Let $(V,\omega,H)$ satisfy the following conditions:
\begin{enumerate}
\item
$(V,\omega,H)$ satisfy Assumption (A);
\item
$(V,\omega)$ is a $c=26$ holomorphic conformal vertex algebra;
\item
The maximal lattice of $(V,H)$ is isomorphic to $\genus_{17,1}(2_{\genus}^{+10})$.
\end{enumerate}
Then,  $(V,\omega,H)$ is isomorphic to $V_{\genus_{17,1}(2_{\genus}^{+10})}^{hol}$.
\end{thm}
\begin{proof}
Consider an isomorphism of lattices $\genus_{17,1}(2_{\genus}^{+10})$ and $\sqrt{2}E_8D_8\oplus \tw$. Similarly to the proof of Proposition \ref{coset_correspondence}, let $C$ be the coset vertex algebra correspondence to this decomposition.
The vertex algebra $C$ is a $c=24$ holomorphic VOA and its maximal lattice is $\sqrt{2}E_8D_8$.
By Theorem \ref{holomorphic} and Lemma \ref{lem_level1} and Theorem \ref{unique_hol},
$C$ is isomorphic to $V_{E_{8,2}B_{8,1}}^{hol}$ as a VOA and also as a VH pair, by Lemma \ref{unique_Cartan}.
Hence, $V \cong V_{E_{8,2}B_{8,1}}^{hol} \otimes V_\tw$ as a conformal vertex algebra and VH pair.
\end{proof}

\begin{lem}
\label{18_auto}
$G_{V_{\genus_{17,1}(2_{\genus}^+10)}^{hol}}=\Aut\,\genus_{17,1}(2_{\genus}^{+10})
$ and 
$\#genus(V_{E_{8,2}B_{8,1}}^{hol})=17$ holds.
\end{lem}

\begin{proof}
By Lemma \ref{orbit_group} and the proof of Theorem \ref{char_18},
it suffices to show that\\ $\Aut\,(\sqrt{2}E_8D_8) \subset G_{V_{E_{8,2}B_{8,1}}^{hol}}$.
Clearly, $\Aut\,(\sqrt{2}E_8D_8)$ is a semidirect product of Weyl group $W_{E_8} \times W_{D_8}$ and the order $2$ Dynkin diagram automorphism of $D_8$, which is isomorphic to the Weyl group $W_{E_8} \times W_{B_8}$.
By Corollary \ref{affine_reflection}, $G_{V_{E_{8,2}B_{8,1}}^{hol}}$ contains it.
Hence, the assertion holds.
\end{proof}

Let $V$ be a $c=24$ holomorphic VOA such that the maximal lattice is contained in
 $\genus_{16,0}(2_{\genus}^{+10})$.
Since the maximal lattice of $V \otimes V_\tw$ is $\genus_{17,1}(2_{\genus}^{+10})$,
by Theorem \ref{char_18}, we have $V_{\genus_{17,1}(2_{\genus}^+10)}^{hol} \cong V \otimes V_\tw$.
Hence, $V$ is contained in $genus(V_{E_{8,2}B_{8,1}}^{hol})$.
By Lemma \ref{18_auto}, we have:

\begin{prop}
Let $V$ be a $c=24$ holomorphic VOA.
If the maximal lattice $L_V$ of $V$ is contained in $\genus_{16,0}(2_{\genus}^{+10})$,
then $V$ is uniquely determined as a VOA by $L_V$ and $G_{V,\omega}=\Aut L_V$.
In particular, there is a one-to-one correspondence between $\gen(V_{E_{8,2}B_{8,1}}^{hol})$ and $\gen(\sqrt{2}E_8D_8)$.
\end{prop}
\begin{rem}
The one-to-one correspondence between $genus(V_{E_{8,2}B_{8,1}}^{hol})$ and $genus(\sqrt{2}E_8D_8)$ first appeared in
\cite{HS1}, as it was pointed out in \cite{HS2}. Our approach is motivated by their results.
\end{rem}

\section*{Acknowledgements}
The author would like to express his gratitude to Professor Atsushi Matsuo, for his encouragement throughout this work and numerous advices to improve this paper.
He is also grateful to Hiroki Shimakura and Shigenori Nakatsuka for careful reading of this manuscript and their valuable comments. This work was supported by the Program for Leading Graduate Schools, MEXT, Japan.

%

\end{document}